\documentclass{article}

\usepackage{PRIMEarxiv}
\usepackage[utf8]{inputenc} 
\usepackage[T1]{fontenc}    
\usepackage{hyperref}       
\usepackage{url}            
\usepackage{booktabs}       
\usepackage{amsfonts}       
\usepackage{nicefrac}       
\usepackage{microtype}      
\usepackage{lipsum}
\usepackage{fancyhdr}       
\usepackage{graphicx}       
\graphicspath{{media/}}     
\usepackage{amsmath}
\usepackage{amsthm}
\usepackage{cases}
\usepackage{enumitem}
\usepackage{bm}
\usepackage[dvipsnames]{xcolor}
\usepackage{amssymb}
\usepackage{tikz}
\usetikzlibrary{arrows.meta}
\newtheorem{definition}{Definition}
 \usepackage{comment}
\usepackage{subfiles}

\usepackage{bbm}

\newtheorem{theorem}{Theorem}
\newtheorem{corollary}[theorem]{Corollary}
\newtheorem{lemma}[theorem]{Lemma}
\newtheorem{proposition}[theorem]{Proposition}
\newtheorem{conjecture}[theorem]{Conjecture}

\hypersetup{
    colorlinks=true,
    linkcolor=blue,
    filecolor=cyan,      
    urlcolor=red,
    citecolor = red
    }

\DeclareMathOperator*{\argmax}{arg\,max}

\theoremstyle{definition}
\newtheorem{examplex}{Example}
\newenvironment{example}
  {\pushQED{\qed}\examplex}
  {\popQED\endexamplex}

\title{Scale-free cascading failures\\
\large Generalized approach for all simple, connected graphs

}

\author{
  Agnieszka Janicka$^1$, Fiona Sloothaak$^{1}$, and Maria Vlasiou$^{2}$ \\
  $^1$\textit{Eindhoven Univeristy of Technology, 5612 AZ Eindhoven, the Netherlands}\\
  $^2$\textit{University of Twente, 7522 NB Enschede, the Netherlands}\\
}

\begin{document}
\maketitle

\begin{abstract}
    Cascading failures, wherein the failure of one component triggers subsequent failures in complex interconnected systems, pose a significant risk of disruptions and emerge across various domains. Understanding and mitigating the risk of such failures is crucial to minimize their impact and ensure the resilience of these systems.
    In multiple applications, the failure processes exhibit scale-free behavior in terms of their total failure sizes. Various models have been developed to explain the origin of this scale-free behavior. A recent study proposed a novel hypothesis, suggesting that scale-free failure sizes might be inherited from scale-free input characteristics in power networks.
    However, the scope of this study excluded certain network topologies.
   Here, motivated by power networks, we strengthen this hypothesis by generalizing to a broader range of graph topologies where this behavior is manifested.
    Our approach yields a universal theorem applicable to all simple, connected graphs, revealing that when a cascade leads to network disconnections, the total failure size exhibits a scale-free tail inherited from the input characteristics. We do so by characterizing cascade sequences of failures in the asymptotic regime. 


\end{abstract}

\keywords{Cascading failures \and Heavy-tailed failures \and Scale-free failures \and Overload failures \and Principle of a single big jump}
{\renewcommand{\arraystretch}{1.25}%

\section{Introduction} \label{intro}

A cascading failure is a dynamic process that occurs within complex networks composed of interconnected parts, where the failure of an initial component triggers a chain reaction, causing subsequent failures of successive parts \cite{DESHKAR2016, Mahmoud2019}. Cascading failures have the potential to propagate throughout the entire network, resulting in significant disruptions and disturbances for its users. Consequently, the study of cascading failures aims to understand their emergence, analyze their behavior, and develop effective strategies to mitigate associated risks. 
Prominent instances of cascading failures span diverse domains, including but not limited to power grid blackouts \cite{Guo2017}, the spread of epidemics \cite{Mahdi2017} or viruses in computer systems \cite{Crucitti2004}, congestion in transportation systems \cite{Daqing_2014}, disruption of supply chains \cite{TANG201658}, systemic risk in finance~\cite{Huang_Vodenska_Havlin_Stanley_2013}, and various other domains \cite{Watts_2002}.


In many such cascading networks, it has been observed that failure sizes, which can be the blackout size, the number of individuals affected by the epidemics, or congestion delay,  have heavy tails \cite{barabasi1999emergence, chen2001analysis, Zhang_Zeng_Li_Huang_Stanley_Havlin_2019}. Specifically, tails are often \textit{scale free}, which means that
\begin{equation}\label{totalFailureSize}\mathbb{P}(S>x)\sim Cx^{-\alpha},
\end{equation} 
where  $C>0$ and $\alpha>0$ are constants, and $f(x)\sim g(x)$ denotes that $\frac{f(x)}{g(x)}\rightarrow 1$ as $x\rightarrow \infty$.
This heavy-tailed nature implies that the occurrence of significantly large failures, which can be considered ``catastrophic'', cannot be dismissed as virtually impossible. Hence, understanding cascading failure behavior and devising strategies to prevent large-sized failures becomes imperative. With this motivation in mind, numerous studies have developed cascading failure models aiming to explain the emergence of scale-free behavior \cite{Guo2017}. In this paper, we are motivated by applications in power networks. In this setting, the cascading failures lead to blackouts and it is well-established that the distribution of blackout sizes, measured in the amount of lost power, exhibits a scale-free behavior \cite{chen2001analysis}.

Studies of cascading failures often focus on the topological aspect, i.e., to understand how the underlying network influences the spread of failures. Large research effort in this field is devoted to the analysis of \textit{structural} failures \cite{EpidemicCN2014, Pastor2001, Manzano2013, VC2011}, which refers to the mechanism where the failure spreads locally through direct neighbors. Cascading failures in power networks, which motivate the study in this paper, exhibit an intrinsically different behavior. If a line exhibits a failure, the next line to trip may be disconnected from the original source \cite{Guo2017}. This type of mechanism is called \textit{overload} failures and has also received a lot of attention \cite{CNCFreview}. For instance, in \cite{Motter2002} the authors consider a nodal overload model on \textit{scale-free networks} in which some relevant quantity is transported between nodes via the shortest path. In the context of complex networks, the term \textit{scale-free networks} refers to networks with a scale-free degree distribution \cite{hofstad2016}. Other models with similar characteristics under various performance measures are studied in \cite{Crucitti2004}, \cite{Latora2001}, and \cite{Lai2004}.
The models in \cite{Wang_2008,WANG20086671,WANG20091289,WANG20091332} are motivated by power networks, which is why they assume that the load on the network is redistributed in a manner dictated by physical laws. The objective of these studies is to understand the influence of the network's structure on the vulnerability of the system, with a particular focus on scale-free networks. As such, the findings are induced by the assumed network structure and do not generalize to other network topologies. In this paper, we show that scale-free failure sizes can emerge in other network structures as well.

 The study of scale-free failures has expanded beyond the topological aspect, providing other hypotheses for the emergence of scale-free failures. For example, the Sandpile model, proposed in \cite{SOC1987}, introduces the notion of \textit{Self-Organized Criticality} (SoC) as a potential emergence mechanism. This notion assumes that a system comprising particles naturally evolves towards  \textit{criticality} -- a stage where even a small disturbance can lead to a radical behavioral change of the system, called a \textit{phase transition}. This transition induces a redistribution of particles, leading to a cascade. For example, in the context of power networks, a \textit{critical} stage occurs when the system operates at its full capacity, and a \textit{phase transition} can be caused by a local disturbance, such as a failure of a line, leading to a cascade of line failures. The authors show that the \textit{phase transition} in self-organized systems leads to scale-free failure sizes. Models inspired by SoC have been applied to study phenomena in physics, biology, computer science or geology~\cite{PhysRevLett.69.1629,Beggs_Plenz_2003,PhysRevLett.68.1244, VALVERDE2002636}. For a comprehensive account of the SoC theory, we refer to \cite{SoCbook}. 

Specifically for power networks, which is the focus of our work, the notion of SoC has been applied to explain scale-free blackout sizes. For example, in \cite{Dobson2000} the authors present initial evidence for SoC in power networks, using time-series analysis methods. The cascade size in this and other SoC-inspired models exhibits a heavy tail \cite{Dusco2006}, which means that SoC could be a potential explanation for the scale-free nature of blackout sizes. Another notable model in this field is the OPA model \cite{OPA2002}. Similar to the Sandpile model, the redistribution of power in the OPA model can lead to line overloads and subsequent failures, creating a cascade. The model assumes two time scales, `slow' and `fast'. The `slow' time scale is used to model the continuous increase of power demand over time. This increase drives the system to operate near its full capacity. On the other hand, cascading failures happen in the `fast' time scale. A cascade may be triggered either by random line failures or when a maximum capacity has been reached due to a demand increase. The authors use a standard approach in modeling power networks by using a linear approximation of true power flow, called the \textit{DC power flow} equations \cite{OPA2002, DCPowerFlowRevisited}, which we also apply in this paper.  

Other than Self-Organized Criticality, many other mechanisms have been proposed to explain the emergence of scale-free blackouts. An example of such a model is the CASCADE model \cite{Cascade2003}. It is a probabilistic, overload failure model that can exhibit heavy-tailed failure size, but only in a critical regime. The authors show that the total failure size is given by a quasi-binomial distribution. This distribution can be approximated by the offspring distribution of a branching process, which, in a critical state, is heavy-tailed with parameter $\alpha = \frac{1}{2}$ as in \eqref{totalFailureSize}.

While this literature offers many different approaches for the analysis of blackouts in power networks, the emergence of scale-free blackout sizes is still not well understood. The SoC-inspired models exhibit the desired behavior; however, there is not enough statistical evidence showing that real power networks are indeed self-organized~\cite{Watkins_Pruessner_Chapman_Crosby_Jensen_2015}. The CASCADE model is highly stylized and the load distribution mechanism is too simplistic to capture the behavior of power networks. Hence, it does not offer a credible explanation of the emergence of scale-free blackouts. The research in complex networks offers explanations for particular graph structures, but these findings do not generalize to other topologies. Thus, the study of cascading failures in power networks continues to be a flourishing area of research. 

Our work enriches the literature on the emergence of scale-free cascading failures by strengthening a recent result. Our approach is inspired by a paper by Nesti et al.\,\cite{Nesti_2020}, which presents a novel probabilistic explanation for the emergence of scale-free blackouts. The authors propose a stochastic overload model and, under certain assumptions, derive results for the tail of the blackout distribution. The results suggest that scale-free blackout sizes may be induced by the scale-free demand distribution. In this paper, we adopt the model of \cite{Nesti_2020} and generalize the main result of the paper. By taking an alternative proof approach, we remove some of their assumptions. Thus, we obtain a universal result applicable to all connected and simple graph topologies i.e. graphs with no multi-edges or self-loops, thereby advancing our understanding of cascading failures. 


The model in \cite{Nesti_2020} assumes a power network with $n$ nodes and $m$ edges. Each node $i\in \{1,\dots,n\}$ has a specified demand $d_i$ and generation $g_i$. The nodal demand $d_i$ is a realization of the Pareto-tailed demand random variable $X_i$ with parameter~$\alpha$, where $X_i$'s are mutually independent. In the context of power networks, this assumption is justified by the well-established observation that city-size distributions are scale free \cite{ROSEN1980165, Beckmann_1958, Berry1961}. The model assumes that the demand is exogenous and the generation is planned to meet this demand. In particular, the generation is a solution to a convex direct-current optimal power flow problem (DC-OPF) \cite{DCformulation}. While individual nodes may not maintain balance, the total generation and demand in the network are equal, i.e., $\sum_{i = 1}^n d_i - g_i = 0$. This is achieved by transporting power between nodes through the edges according to the DC power flow equations. Each edge has a capacity limit that is initially not exceeded. A cascade of failures is triggered by an external disturbance, leading to the failure of an edge. If the failure leads to network disconnection, the demand-generation balance is restored in each connected component by proportionally decreasing whichever resource is in abundance. Consequently, the flow from the failed edge is redistributed. The edge with the largest \textit{relative exceedance} (if any) -- the absolute ratio of flow and capacity -- fails next and the flow is redistributed again.  This process continues until there are no more flows on edges that exceed the edge capacity, marking the end of the cascading failure process.

The main object of analysis in \cite{Nesti_2020} is the total cascade size $S$. It is defined as the total demand lost due to balance restoration. Specifically,
\[S := \sum_{i = 1}^n \left(d_i - d_i^{(\text{end})}\right),\]
where $d_i^{(\text{end})}$ is the demand of node $i$ after the cascade stops. The authors show that, under certain assumptions, the total cascade size inherits the Pareto tail from the demand distribution, i.e., Equation \eqref{totalFailureSize} holds whenever network disconnection occurs. In the proof, the authors rely on the principle of a single big jump; they show that a large cascade is predominantly caused by 
scenarios where one node has a significantly larger demand, compared to other nodes. Specifically, for every $\varepsilon>0$,
\begin{equation}\mathbb{P}(S>x)\approx \mathbb{P}\left(S>x; \sum_{\substack{X_i\neq M, \\i \in \{1,\dots,n\}}}X_i< \varepsilon M\right),\label{SBJ}\end{equation}
with $M = \max\{X_1,\dots,X_n\}$ and $x>0$ large. Furthermore, the authors show that the cascade sequence remains the same for all $\varepsilon\geq 0$ sufficiently small. As a consequence, to obtain the tail behavior of $S$, it is sufficient to understand the total cascade size in the limiting case of $\varepsilon = 0$, which is exploited in the proof.

 The results in \cite{Nesti_2020} rely on two assumptions, ensuring that the cascade is well-defined. Our first contribution is showing that these assumptions are not merely technical but essential, as not all graphs satisfy them. More specifically, Theorem IV.12 in \cite{Nesti_2020} requires a unique maximizer of the exceedance at every step of the cascade. In Lemma \ref{conditions}, given in Section \ref{emergency}, we derive necessary and sufficient conditions that characterize this uniqueness. Then, based on this insight, in Example \ref{example}, we construct a graph for which there are non-unique maximizers during the cascade. This raises the need for an ``exceedance tie-breaking rule'', specifying which edge should fail when a non-unique maximizer is encountered. With such a rule in place, the cascade of failures is well-defined for any simple, connected graph.

After describing the model in Section \ref{model}, we give in Section~\ref{mainResults} our second and main contribution, Theorem \ref{mainThm}. The theorem states that for a simple, connected graph $\mathcal{G}$ with Pareto-tailed demand vector $d = (d_1,\dots, d_n)$, if a cascade leads to network disconnection, then the total failure size has a Pareto tail (see Equation~\eqref{totalFailureSize}). Note that Theorem IV.12 of \cite{Nesti_2020} is a special case of our theorem, which we show in Lemma \ref{equivThm} in Section \ref{impact}, where we also discuss the impact of the "exceedance tie-breaking rule" on our results. 
To prove the theorem, in the first step, we rely on the principle of a single big jump, as given in Equation \eqref{SBJ}, to show that we can limit our analysis to scenarios when one node has a large demand and the sum of other demands is bounded. Specifically, we consider all demands satisfying $\sum_{i\neq j} d_i<\varepsilon d_j$ for fixed $\varepsilon>0$ and $j\in \{1,\dots,n\}$. Observe that even for a fixed node $j$ and $\varepsilon>0$, there are infinitely many demand vectors that satisfy this constraint. This is where our proof strategy diverges from \cite{Nesti_2020}, where the authors solely consider the case of $\varepsilon = 0$. The next step is to obtain the optimal generation vector $g^*(d) = (g_1^*,\dots,g_n^*)$, given as the solution of the DC-OPF problem with input $d$ (see  Equation \eqref{DC-OPF}). To the best of our knowledge, the closed-form solution for $g^*(d)$ is not known. Hence, in Lemma \ref{AltOptSol} we characterize the optimal solution $g^*(d)$, showing it is a projection onto a facet of a well-defined convex polytope $\mathcal{F}_d$, given by the feasible region of the DC-OPF problem. Exploiting this characterization, in Proposition \ref{propProjection} we demonstrate that the demands can be partitioned into a \textit{finite} number of subsets, each corresponding to a projection onto a different facet of $\mathcal{F}_d$. Consequently, in Proposition \ref{propositionKKT}, we show that the cascade sequence does not depend on $\varepsilon$ in the neighborhood of 0, which is a technical, yet crucial result required in the proof of Theorem \ref{mainThm}. As a result, all demands in a given subset induce a fixed cascade sequence, which we show in Lemma~\ref{cascadeEpsilon}. Finally, we prove that $S$ exhibits a scale-free tail behavior by conditioning on each of the finitely many cascading sequences and exploiting the independence with respect to $\varepsilon$. All major results are proven in Section \ref{proofs} and proofs of auxiliary results are given in Appendix \ref{appendixProofs}. 

\section{Model and notation} \label{model}

In this section, we describe the model of power networks used in this paper. The basic element of the model is a graph with nodes and edges, where each node has an exogenously given demand. In addition, each node may generate power, which together with the demand, induces power flow through the edges of the graph. Edge flow limits, optimal power generation, and the cascading failure process are determined sequentially in three stages:
\begin{enumerate}
    \item Planning stage. In this stage, we determine \textit{operational} edge flow limits following a basic network design process. First, given nodal demands, we determine nodal \textit{planning} generations that yield the minimal generation costs, assuming that there are no power flow limitations. Then, we choose \textit{planning} edge capacities equal to the \textit{planning} flow, induced by the exogenous demand and planning generation. Finally, we set the \textit{operational} edge flow limits as a fraction $\lambda\in (0,1)$ of the \textit{planning} edge capacities, ensuring that the system can withstand some unexpected flow changes.
    \item Operational stage. This stage describes the standard operation of the system, i.e., before the cascade. First, given the graph and all demands, the optimal nodal generations are determined by minimizing generation costs subject to operational edge flow limits. Given these generations, power is distributed through the edges to meet the demand at each node. By construction, the power flow on the edges does not exceed the operational edge flow limits, hence the system is stable.
    \item Emergency stage. In this stage, the stability of the system is disrupted by a random edge failure, which may induce a cascade of failures in the system. We compute the cascade sequence iteratively, by redistributing the power after each edge failure and consequently identifying the edge with the biggest strain, i.e. the next edge failure, if any.

\end{enumerate}
Given this broad overview of the mechanics of the model, we now proceed to describe the model in detail.
\subsection*{Network characteristics}
As a part of the network characteristics, we need to describe the network, the demands, and the mechanics of the power flow in the system. Regarding the network, consider a simple, connected, directed graph $\mathcal{G} = (\mathcal{N},\mathcal{E})$ with $n$ nodes and $m$ edges, represented by the incidence matrix $C\in \mathbb{R}^{m\times n}$. In particular,
\[C_{e,i} = \begin{cases}
    1 & \text{if edge $e$ enters node $i$},\\ 
    -1 & \text{if edge $e$ leaves node $i$},\\ 
    0 & \text{otherwise}.\\ 
\end{cases}\]
Furthermore, let $\mathcal{E}_i$ denote the set of edges adjacent to node $i$, namely $\mathcal{E}_i := \{e\in \mathcal{E} ~:~ C_{e,i}\neq 0\}$. We emphasize that the direction of the edges is not a property of the underlying network but a modeling choice. As a result, the orientation can be selected in a convenient manner without loss of generality (see Lemma \ref{orientation}). 

Let $d = (d_1,\dots, d_n)\in \mathbb{R}_{\geq 0}^n$ be the demand vector, where $d_i$ denote the demand of node $i\in \mathcal{N}$. The demand $d_i$ is a realization of a random variable $X_i,$ which has a Pareto tail with parameter $\alpha>0$. Specifically, \begin{equation}\mathbb{P}(X_i>x) \approx Kx^{-\alpha},\label{Kconstant}\end{equation} for constant $K$ and large $x$. Moreover, we assume that $X_i$'s are mutually independent. 


Power flows in the network according to physical laws, following the AC power flow equations \cite{ACflow2016}. However, due to their non-linear nature, typically approximations of these equations are used in practice. In this work, we use a standard approximation, the DC approximation \cite{DCPowerFlowRevisited}, which linearizes the AC power flow equations. This approximation scheme neglects the edge losses due to the assumption that the resistance on edges is much smaller than their reactance, leading to negligible conductance \cite{PTDF2017, Seifi_Sepasian_2011}. Moreover, this approach assumes that the voltage angles are small and the voltage profile is flat. 

To describe the flow on each edge using the DC power flow approximation, we first need to introduce the notion of edge susceptance that captures the physical characteristics of edges in the network, and the Power Transfer Distribution Factors (PTDF) matrix $V$. To this end, let $b_e$ denote the susceptance of edge $e$ and define a susceptance matrix $B\in \mathbb{R}^{m\times m}$ as $B:=\text{diag}(b_1,\dots,b_m)$. The matrix $V$ is uniquely determined by the network and is given by $V:=BC(C^TBC)^+$, where~$+$ denotes the matrix pseudo-inverse \cite{Nesti_2020}. Given this, the edge flow vector $f\in \mathbb{R}^m$ on graph $\mathcal{G}$ with demand $d$ and generation $g$ is given by \begin{equation}f:= V(d - g)\label{flow}.\end{equation} In this paper, for simplicity, we assume $B = I$, i.e., we assume identical edges. 
\subsection*{DC-OPF formulation}
The DC-OPF model is used in both \textit{planning} and \textit{operational} stages in order to determine the operational flow limit vector $\bar{f}\in \mathbb{R}^m$ and the optimal generation vector $g^*(d)\in \mathbb{R}^n_{\geq 0}$. Hence, we first introduce the DC-OPF program \cite{DCformulation} with a quadratic objective function given below. Recall that $d_i$ and $g_i$ denote the demand and generation of node $i\in \mathcal{N}$, respectively.
\begin{subequations}
 \label{DC-OPF}
\begin{gather}
\min_{g\in \mathbb{R}^n_{\geq 0}} \frac{1}{2}\sum_{i = 1}^n g_i^2:\label{obj}\\
\sum_{i = 1}^n g_i = \sum_{i = 1}^nd_i, \label{balance}\\
-\bar{f}\leq V(d - g)\leq \bar{f}\label{edgeConstraints}.\end{gather}
\end{subequations}
 Observe that this version of DC-OPF is adjusted in the sense that it does not account for generation limits that are typically given as an additional condition. Note that Condition \eqref{balance} ensures the balance between the total generation and demand in the network and Condition \eqref{edgeConstraints} ensures that the flows do not exceed these operational limits. Note that in this formulation, we allow for negative flows; the negative sign indicates that the flow is opposite to the initially chosen direction.
\subsection*{Planning, operational and emergency stage}
We now move on to the description of the stages of the model, beginning with the \textit{planning} stage. It is crucial to note that in order to choose a suitable vector of operational edge limits $\bar{f}$, Problem $\eqref{DC-OPF}$ is solved without taking the edge capacity condition into account, i.e. Condition \eqref{edgeConstraints} is disregarded. This results in the optimal planning generation vector denoted by $g^{(pl)}$ which is unique and equal to $g^{(pl)} = \bar{d}e$ \cite{Nesti_2020}, where $e = (1,\dots,1)\in \mathbb{R}^n$ is the all-ones vector and $\bar{d}:= \frac{1}{n}\sum_{i = 1}^nd_i$ is the average demand per node. Moreover, the corresponding planning flow vector is given by $f^{(pl)} = V(d - g^{(pl)})$. Then, we set the \textit{operational} edge limits to $\bar{f} = \lambda |f^{(pl)}|$, where $\lambda\in (0,1)$ represents a safety tuning parameter called the loading factor. Notably, since $\text{Ker}(V) = \langle e\rangle$ \cite{Nesti_2020}, we deduce that the planning flow vector is given by $f^{(pl)} = Vd$ and the operational edge limits become $\bar{f} = \lambda|Vd|$. This marks the end of the \textit{planning} stage.

In the next stage of the problem, i.e. the \textit{operational} stage, we determine the operational generation vector $g^*(d) = (g^*_1,\dots,g^*_n)\in \mathbb{R}^n_{\geq 0}$. The generation vector $g^*(d)$ is a function of the demand vector $d$ and is the unique solution to the DC-OPF problem \eqref{DC-OPF}. Hence, in the operational stage, Program \eqref{DC-OPF} is solved again, now with Condition \eqref{edgeConstraints} in place, which yields $g^*(d)$. With the exogenous demand $d$ and generation $g^*(d)$ the system operates in equilibrium until an external disturbance occurs, leading to an edge failure and initiating the \textit{emergency} stage.

At the start of the \textit{emergency} stage, an edge fails uniformly at random, resulting in a modified graph $\mathcal{G}^{(2)} = (\mathcal{N}^{(2)},\mathcal{E}^{(2)})$ with corresponding PTDF matrix $V^{(2)}$. This may result in an adapted demand $d^{(2)}$ and generation $g^{(2)}$. Specifically, the demand and generation may be modified if the removal of the failed edge leads to a disconnected graph. This can cause an imbalance of generation and demand in the connected components. In such a case, we restore it by proportionally reducing either demand or generation in each node within each component according to Equation \eqref{balanceRestoration} given in the next paragraph with input $r = 1$, $d^{(1)} = d$, and $g^{(1)} = g^*(d)$. In the modified graph, the flow vector is given by $f^{(2)} = V^{(2)}(d^{(2)} - g^{(2)})$. 

In order to determine which edge will break next, we define the \textit{emergency} edge limit \begin{equation}F:= \frac{\lambda^*}{\lambda}\bar{f}\label{emergencyLimit}\end{equation} for some $\lambda^*\geq \lambda$. If $\lambda^* = 1$, we can interpret the emergency limit $F$ as the true capacity limit of an edge, while the operational limit $\bar{f}$ is an additionally imposed safety feature ensuring that the system does not operate at its full capacity in the operational stage. To determine the next edge to break, we look at the ratio of edge flow and the emergency edge limit. Specifically, if any component of $f^{(2)}$ exceeds its corresponding component in $F$, i.e. the flow exceeds the emergency edge limit for some edge, the cascade effect continues. The next edge to break is the one with the maximum relative exceedance $\psi_e^{(2)}:=\left|f^{(2)}_e\right|/F_e$, with edge $e$ determined as  
\begin{equation}\argmax_{e\in \mathcal{E}^{(2)}}\{\psi_e^{(2)}:\psi_e^{(2)}> 1\}.\label{psi}\end{equation}
The subsequent steps of the cascade occur in the same manner. Suppose that in the $r$-th step of the cascade, $r\in \mathbb{N}$, an edge has just failed. This results in graph $\mathcal{G}^{(r+1)} =(\mathcal{N}^{(r+1)},\mathcal{E}^{(r+1)})$ with the PTDF matrix $V^{(r+1)}$. The demand vector $d^{(r+1)}$ and the generation vector $g^{(r+1)}$ are obtained from $d^{(r)}$ and $g^{(r)}$ through the process of balance restoration. Specifically, for a connected component $\mathcal{C} = (\mathcal{N}_\mathcal{C}, \mathcal{E}_\mathcal{C})$ of $\mathcal{G}^{(r+1)}$, let $\theta^{(r+1)} = \sum_{i\in \mathcal{N}_\mathcal{C}} d_i^{(r)}/\sum_{i\in \mathcal{N}_\mathcal{C}} g_i^{(r)}$ denote the ratio of total demand and generation in this component. Then, for $i\in \mathcal{N}_\mathcal{C}$,
\begin{equation} \label{balanceRestoration} \left(d^{(r+1)}_i,~ g^{(r+1)}_i\right) = \begin{cases}
 \left(\frac{1}{\theta^{(r+1)}} d^{(r)}_i,~ g^{(r)}_i \right)&\text{if }\theta^{(r+1)} \geq 1,\\
\left(d^{(r)}_i,~\theta^{(r+1)}g^{(r)}_i\right) &\text{if }\theta^{(r+1)} \leq 1.
\end{cases}\end{equation} This process ensures that if there is excess of demand in the component, i.e., $\theta^{(r+1)}>1$, the nodal demands are lowered by the factor $1/\theta^{(r+1)}$, restoring the balance of total generation and demand. Otherwise, the nodal generations are lowered in a corresponding manner. Lastly, the flow vector $f^{(r+1)}$ is given by $V^{(r+1)}(d^{(r+1)} - g^{(r+1)})$. If the cascade continues, the next edge to break is the maximizer of the relative exceedance $\psi_e^{(r+1)}:= |f_e^{(r+1)}|/F_e$ with $e$ determined similarly as in \eqref{psi}. 


    In the model by Nesti et al.\,\cite{Nesti_2020}, all graphs displaying non-uniqueness of flow exceedance maximizers are excluded from the analysis. Therefore, in order to analyze all simple, connected graphs, we must define how the cascade sequence should proceed when a non-uniqueness of a maximizer is encountered. More specifically, we introduce a fixed ``tie-breaking rule'' that determines which edge(s) should break next in the event of a non-unique exceedance maximizer. One commonly used tie-breaking rule in cascade models is to simultaneously break all edges that maximize the exceedance during a single cascade step, given that the exceedance is greater than 1. Alternatively, we may choose the next edge to fail based on the labeling of the maximizes, such as always breaking the edge with the smallest label. We refrain from imposing any specific tie-breaking rule, but we do require that the ties are resolved according to some deterministic tie-breaking rule $T$, defined as follows.
    
    \begin{definition} \label{tie-breakingDef}
        A deterministic tie-breaking rule $T$ is a function that takes as input a subset of edges $\mathcal{M}\subseteq \mathcal{E}$, representing the set of maximizers of the relative exceedance during a single cascade step. It outputs a non-empty subset $\overline{\mathcal{M}}\subseteq \mathcal{M}$, representing the set of edges that break during the current step of the cascade. Formally, $T: 2^{\mathcal{E}}\rightarrow 2^{\mathcal{E}}$ such that $T(\mathcal{M})\neq \emptyset$ is the set of edges selected to fail. Note that if $|\mathcal{M}| =1$, meaning that there is only one maximizer, then $\overline{\mathcal{M}} = \mathcal{M}$. 
    \end{definition}
\noindent Notably, this definition of a tie-breaking rule excludes any randomized approach, such as breaking one of the maximizer edges uniformly at random. Nevertheless, it is worth mentioning that the results presented in this paper can be extended in a straightforward manner to accommodate randomized tie-breaking rules as well. 

The cascade stops if there are no more edges exceeding the emergency limit. The resulting graph and demand vector are denoted respectively by $\mathcal{G}^{(\text{end})}$ and $d^{(\text{end})}$. The total failure size $S$ is defined as the total demand lost during the process of balance restoration, i.e. $S = \sum_{i = 1}^n d_i - d^{(\text{end})}_i$. If the cascade leads to no network disconnection, then $S = 0$.
\subsection*{Additional assumptions and notation}
With the model fully specified, we proceed to describe additional notation and assumptions used throughout this paper. Next, let $X := (X_1,X_2,\dots, X_n)$ be a vector of Pareto-tailed demands with parameter $\alpha$. We define $Y$ as the reordering of $X$ such that the maximum demand is shifted to the first position; that is, 
\begin{equation} \label{Y}Y = (X_{i}, X_1,X_2,\dots,X_{i-1},X_{i+1},\dots,X_n), \text{ with } i:= \argmax_{i = 1,\dots,n}\{X_i\}.\end{equation}
This notation ensures that the node with the largest demand always has label $1$ in $Y$. Furthermore, in our analysis, we typically focus on specific normalized demand realizations $d(\varepsilon,\gamma)$ given by
\begin{equation} \label{genericD} d(\varepsilon,\gamma) := e_1 + \varepsilon\gamma, \text{ for some } \varepsilon>0 \text{ and } \gamma\in \Gamma,\end{equation} 
where $\Gamma := \{\gamma = (\gamma_1,\dots,\gamma_n) ~|~ \gamma_1 = 0, \gamma\geq 0, e^T\gamma = 1\}$ and   $e_1 = (1,0,\dots,0) \in \mathbb{R}^n$. This particular choice is motivated by the principle of a single big jump, as explained in the introduction. Namely, this choice assumes that the first node has a significantly larger demand, compared to other nodes. Note that this is a normalized demand profile, as it ensures that the largest nodal demand is equal to 1. In other words, $d(\varepsilon,\gamma)$ corresponds to a realization $y = (y_1,\dots, y_n)$ of $Y$ scaled by $y_1$, where $\frac{1}{y_1}(0,y_2,\dots,y_n) = \varepsilon \gamma $ for some vector $\gamma\in \Gamma$. Here, $\gamma$ represents the ratios between all but the first component of $y$, and $\varepsilon$ represents the proportion of demand between the node with the largest demand and all other nodes. 

For convenience of exposition and calculations, it is suitable to choose a specific orientation of the edges. In what follows, we choose the orientation induced by the following lemma.

\begin{lemma} \label{orientation}
Given a graph $\mathcal{G}$ and $\gamma\in \Gamma$, it is possible to adjust the orientation of edges in the graph so that $Vd \geq 0$ for all $\varepsilon\geq 0$ sufficiently small, where $d = e_1 + \varepsilon \gamma$.
\end{lemma}
\begin{proof}
Let $i \in \mathcal{E}$ and let $d$ be a demand vector as given in \eqref{genericD}. Suppose that $(Ve_1)_i>0$. Then,  $(Vd)_i = (Ve_1)_i + \varepsilon(V\gamma)_i$. For $\varepsilon < \frac{(Ve_1)_i}{|(V\gamma)_i|}$, we observe that $\varepsilon|(V\gamma)_i|<(Ve_1)_i$. Therefore, for this choice of $\varepsilon$, we find that $(Vd)_i\geq 0$.

Now, suppose that $(Ve_1)_i=0$ and $(Vd)_i<0$. Then, we swap the orientation of edge $i$. This operation results in swapping signs in the $i-$th row of and $V$, and it does not affect the other edges. Since $(Ve_1)_i = 0$, it remains $0$ after swapping, and now $(Vd)_i>0$. Both cases together imply that we can assume $Ve_1\geq 0$ and $Vd \geq 0$, given $
\varepsilon$ sufficiently small. This completes the proof.\end{proof}
\noindent Note that we can assume the orientation induced by Lemma \ref{orientation} without loss of generality because the orientation of edges influences only the sign of the flow, which does not influence the cascade of failures.

In our analysis, the optimal generation $g^*(d)$ will be characterized as a projection onto the boundary of the feasible region of the DC-OPF problem. To this end, let $\mathcal{F}_d$ denote the feasible region defined by the Conditions \eqref{balance} and \eqref{edgeConstraints} of the DC-OPF problem:
\[\mathcal{F}_d:= \{g\in \mathbb{R}_{\geq0}^{n}:e^Tg = e^Td,~~ -\lambda|Vd|\leq V(d - g)\leq \lambda |Vd|\}.\]
 Under the orientation of edges induced by Lemma \ref{orientation}, the inequality constraints of the DC-OPF problem simplify to $(1-\lambda)Vd\leq Vg\leq (1+\lambda)Vd$. Therefore, $\mathcal{F}_d$ can be written as
 \begin{equation}\mathcal{F}_d= \{g\in \mathbb{R}_{\geq0}^{n}: e^Tg = e^Td, ~~ (1-\lambda)Vd\leq Vg\leq (1+\lambda )Vd\}.\label{Fd}\end{equation}
Observe that $\mathcal{F}_d$ is constrained by $2m+1$ hyperplanes, which are given by
\[H_0:= \{g\in \mathbb{R}^n: e^Tg = e^Td\},\]
\[H_i:= \{g\in \mathbb{R}^n: (1-\lambda)(Vd)_i = (Vg)_i\},~~ i = 1,\dots,m,\]
\[H_j:= \{g\in \mathbb{R}^n: (1+\lambda)(Vd)_{j-m} = (Vg)_{j-m}\},~~ j = m+1,\dots,2m.\]
Next, let $\mathcal{I}:=2^{\{0,\dots, 2m\}}$ be the power set of all hyperplane indices and let the hyperplane index set $I^*(d)\in \mathcal{I}$ be the subset of indices of hyperplanes containing the optimal solution $g^*(d)$. Specifically, 
\begin{equation}\label{Istar}
    I^\ast(d):= \{i\in \{0,\dots,2m\} : g^*(d)\in H_i\}.
\end{equation}We now define the affine subspace $S_{I^*(d)}$ where the optimal solution $g^*(d)$ lies, which is given by
\begin{equation*}S_{I^\ast(d)}:=\bigcap_{i\in I^*(d)} H_i.\end{equation*}
Observe that as we show in the proof of Lemma \ref{lemmaProjection}, $g^*(d)$ is a unique point on the boundary of $\mathcal{F}_d$ and therefore $S_{I^*(d)}$ is non-empty.
Later on in the proofs of our results, we use an alternative characterization of $g^*(d)$ through a projection matrix $A_{I^*(d)}$ defined in Lemma \ref{lemmaProjection}, where we show that $$g^*(d) = A_{I^*(d)}d.$$ Intuitively, $A_{I^*(d)}$ projects the demand onto the affine space $S_{I^*(d)}$. This characterization allows us to exploit linear properties of $g^*(d)$. Moreover, in Proposition \ref{propositionKKT}, we show that for all $\varepsilon>0 $ small, $I^*(d)$ depends on $d$ only through $\gamma$. In particular, there exists a mapping $\phi: \Gamma\rightarrow \mathcal{I}$ such that $I^*(d) = \phi(\gamma)$ for all $\varepsilon>0$ small. Note that we use the notation $I^*\in \mathcal{I}$ when we refer to some element of $\mathcal{I}$ and not the particular element corresponding to demand $d$. Observe that since $\mathcal{I}$ is a finite set, there are finitely many matrices $A_{I^\ast}$ with $I^*\in \mathcal{I}$.

Next, we give additional notation used in the analysis of the total cascade size. Let the random variable $N\in \mathcal{N}$ be the label of the node with the largest demand and let $L^{(1)}\in \mathcal{E}$ be the random variable denoting the edge that failed first. Since all nodal demands are i.i.d., each node has the same probability of having the largest demand; i.e., $\mathbb{P}(N = i) = \frac{1}{n}$ for all $i\in \mathcal{N}$. Moreover, the first edge failure is chosen uniformly at random; hence, $\mathbb{P}(L^{(1)} = l) = \frac{1}{m}$ for all $l\in \mathcal{E}$. Furthermore, let $Z(i,l,\gamma)$ be the number of nodes disconnected from node $N = i$ after the completion of the cascade that was initiated by the failure of edge $L^{(1)} = l$ and initial demand $d(\varepsilon,\gamma)$. In our analysis, we show that $Z(i,l,\gamma)$ is well defined for all $\varepsilon>0$ sufficiently small, as in this setting the cascade depends on $d(\varepsilon,\gamma)$ only through $\gamma$. Note that $Z(i,l,\gamma)$ takes values in $\mathcal{Z}:=\{0,\dots, n-1\}$. With this, for a given node with largest demand $i\in \mathcal{N}$, first edge failure $l\in \mathcal{E}$, number of disconnected nodes $z\in \mathcal{Z}$, and hyperplane index set $I^*\in \mathcal{I}$, we define $\Gamma(i,l,z,I^*)$ as the subset of $\gamma\in \Gamma$, for which the hyperplane index set is equal to $\phi(\gamma)= I^*$ and the number of disconnected nodes is equal to $Z(i,l,\gamma) = z$ in the neighborhood of $\varepsilon = 0$. Specifically, 
\begin{equation*}\Gamma(i,l,z,I^\ast) = \{\gamma\in \Gamma ~:~ \phi(\gamma) = I^*, ~~Z(i,l,\gamma)  = z\}.\end{equation*}
What we show next is that subsets $\Gamma(i,l,z,I^\ast)$ partition $\Gamma$. We show in Lemma \ref{cascadeEpsilon} that for a given node $N = i$ and an edge $L^{(1)} = l$, $\gamma$ uniquely defines the number of disconnected nodes $Z(i,l,\gamma) = z$. Further, recall that $\gamma$ uniquely defines the hyperplane index set $I^*$ through $\phi$. As a consequence, for a given $N = i$ and $L^{(1)} = l$, the set $\Gamma$ has the following partition $$\Gamma = \bigcup_{z\in \mathcal{Z}, I^\ast\in \mathcal{I}} \Gamma(i,l,z,I^\ast).$$
Note that the partition is finite because $\mathcal{Z}$ and $\mathcal{I}$ are finite sets.

 Next, we divide possible realizations of the random variable $Y$ (see Equation \eqref{Y}) into subsets $Y(i,l,z,I^*)$. Within each subset, the ratios between components of $Y$ correspond to the ratios $\gamma\in \Gamma(i,l,z,I^*)$. Specifically, \begin{equation}Y(i,l,z,I^\ast):= \{Y~|~ Y_3 = \frac{\gamma_3}{\gamma_2}Y_2, \dots, Y_n = \frac{\gamma_n}{\gamma_2}Y_2, \text{ for some } \gamma = (0,\dots,\gamma_n)\in \Gamma(i,l,z,I^\ast)\}.\label{demandsPartition}\end{equation}
Observe that the set $\Gamma$ denotes all possible ratios between components $i = 2,\dots,n$ of an vector in $\mathbb{R}^n_{\geq 0}$. Hence, for a given $i\in \mathcal{N}$ and $l\in \mathcal{E}$, the subsets $Y(i,l,z,I^*)$ form a finite partition of all possible realizations of $Y$. We exploit this partition in the analysis of the total cascade size by conditioning on the finitely many events $\{Y\in Y(i,l,z,I^*)\}$.

Lastly, we define the $\mathcal{O}$ notation. Consider arbitrary functions $f(x)$ and $g(x)$. Then, $f(x) = \mathcal{O}(g(x))$ as $x\rightarrow \infty$ means that there exists $M> 0$ and $x^\ast\geq 0$ such that $|f(x)| \leq Mg(x)$ for all $x\geq x^\ast$. 

We conclude by summarizing the notation used in this paper in Table \ref{tab:my_label}.

\begin{table}[p]
    \centering
    \resizebox{1\linewidth}{!}{
    \begin{tabular}{|c|l|}
         \hline 
         \textbf{Variable}& \textbf{Description}  \\
         \hline
                 $\mathbb{R}_{\geq0}$ & Set of non-negative real numbers, $\mathbb{R}_{\geq0}:=\{x\in \mathbb{R}~|~x\geq 0\}$.\\
         \hline
         $\mathcal{G} = (\mathcal{N},\mathcal{E})$& Directed graph  of the underlying network with node set $\mathcal{N}$ and edge set $\mathcal{E}$.\\
         \hline
         $\mathcal{E}_i$ & Set of edges adjacent to (either entering or leaving) node $i$.\\
         \hline
         $n, m$ & Number of nodes and edges in graph $\mathcal{G}$, respectively, that is, $n = |\mathcal{N}|$ and $m = |\mathcal{E}|$.\\
                                         \hline
        $C$ & Incidence matrix of $\mathcal{G}$, $C\in\mathbb{R}^{m\times n}$.\\
                                        \hline
        $L$ & Laplacian matrix of $\mathcal{G}$, $L\in\mathbb{R}^{n\times n}$. Specifically, $L = C^TC.$\\
                        \hline
        $V$ & Power Transfer Distribution Factors (PTDF) matrix of $\mathcal{G}$, $V\in \mathbb{R}^{m\times n}$. Specifically, $V = CL^+.$\\
         \hline
        $e$ & Vector of ones of appropriate size, $e = (1,\dots,1)^T$.\\
        \hline
        $e_i$ & $i^{th}$ unit vector of appropriate size.\\
                \hline
         $d,g^\ast(d)$ & Initial vectors of nodal demands and generations, respectively, with  $d, g^\ast(d)\in \mathbb{R}_{\geq0}^{n}$.\\
         \hline 
         $\Gamma$ & Set of all ratios $\gamma$, i.e. $\Gamma := \{\gamma = (\gamma_1,\dots,\gamma_n) ~|~ \gamma_1 = 0, \gamma\geq 0, e^T\gamma = 1\}$.\\
         \hline 
         $d(\varepsilon,\gamma)$ & Particular demand profile $d(\varepsilon,\gamma) := e_1 + \varepsilon\gamma, \text{ for some } \varepsilon>0 \text{ and } \gamma\in \Gamma$.\\
         \hline
         $\bar{d}$ & Average demand per node, $\bar{d} = \frac{1}{n}\sum_{i = 1}^n d_i$. \\
         \hline
         $\lambda$ & Operational safety parameter, $\lambda\in(0,1)$.\\
         \hline
                  $f$ &Vector of edge flows on $\mathcal{G}$ with demand $d$ and generation $g^\ast(d)$, $f\in \mathbb{R}^{m}$. Specifically, $f = V(d - g^\ast(d))$.\\
                           \hline
                  $\bar{f}$ & Vector of operational edge limits on $\mathcal{G}$, $\bar{f}\in \mathbb{R}^{m}$. Specifically, $\bar{f} = \lambda Vd.$\\
                  \hline
                                    $F$ & Vector of emergency edge limits on $\mathcal{G}$, $F\in \mathbb{R}^{m}$. Specifically, $F = \frac{\lambda^*}{\lambda}Vd$ for some $\lambda^*\geq \lambda$.\\

         \hline
         $\mathcal{F}_d$ & Feasible region of the DC-OPF optimization problem; see \eqref{DC-OPF}.\\
         \hline
             
    $H_j$ & $j$-th hyperplane constraining polytope $\mathcal{F}_d$.\\
    \hline
$I^*(d)$ & Hyperplane index set for demand $d$.\\
\hline
$S_{I^*(d)}$ & Face of $\mathcal{F}_d$ containing $g^*(d)$.\\
\hline
$A_{I^*(d)}$& Matrix characterizing the optimal generation, i.e. $g^*(d) = A_{I^*(d)} d$.\\
\hline
$\phi(\gamma)$& Hyperplane index set for $\gamma$ around $\varepsilon = 0$, that is, $\phi(\gamma) = I^*(d)$ for all $\varepsilon>0$ small.\\
\hline

        $\mathcal{G}^{(r+1)}$ & Directed graph of the network after $r$ steps of the cascade with node set $\mathcal{N}^{(r+1)}$ and edge set $\mathcal{E}^{(r+1)}$. \\
        \hline
        $\mathcal{G}^{(\text{end})}$ & Directed graph of the network at the end of the cascade. \\
        \hline
         $d^{(r+1)}$ & Vector of nodal demands after $r$ steps of the cascade, $d^{(r+1)}\in \mathbb{R}^{n}_{\geq 0}$, $d^{(1)}=d$.\\
         \hline
         $d^{(\text{end})}$ & Vector of nodal demands at the end of the cascade, $d^{(\text{end})}\in \mathbb{R}^{n}_{\geq 0}$.\\
         \hline
         $g^{(r+1)}(d)$ & Vector of nodal generations after $r$ steps of the cascade, $g^{(r+1)}(d)\in \mathbb{R}^{n}_{\geq 0}$, $g^{(1)}(d) = g^\ast(d)$.
         \\
         \hline
                          $V^{(r+1)}$ & PTDF matrix of $\mathcal{G}$ after $r$ steps of the cascade, $V^{(r+1)}\in \mathbb{R}^{m\times n}$.\\
                          \hline 
                                            $f^{(r+1)}$ &Vector of edge flows on $\mathcal{G}$ after $r$ steps of the cascade, $f^{(r+1)}\in \mathbb{R}^{m}$. \\
                           \hline
                           $\psi^{(r+1)}$ & Vector of relative edge flow exceedances after $r$ steps of the cascade, $\psi^{(r+1)}\in \mathbb{R}^m$.
                          \\
                          \hline
                          $T$ & Deterministic tie-breaking rule; see Definition \ref{tie-breakingDef}.\\
                          \hline
    $X$ & Vector of i.i.d. Pareto distributed random variables with parameter $\alpha>0$, $X\in \mathbb{R}_{\geq0}^{n}$.\\
    \hline
    $Y$ & Reordering of $X$ by shifting the maximum demand to the first position; see \eqref{Y}.\\
    \hline
    $S$ & Total failure size.\\ \hline

    $N$ & Node with the largest demand, $N\in \mathcal{N}$.\\
    \hline
    $L^{(1)}$ & First edge to fail in the cascade, $L^{(1)}\in \mathcal{L}$.\\
    \hline
    
    $\mathcal{I}$ & Power set of $\{0,\dots, 2m\}$. \\
\hline
$Z(i,l,\gamma)$ & Number of nodes disconnected from node $i$ after edge $l$ initiated the cascade with nodal demand $d(\varepsilon,\gamma)$.\\ \hline 
$\mathcal{Z}$ & Set of the number of nodes disconnected from the node with the largest demand, i.e., $\mathcal{Z} = \{0,\dots, n-1\}$.\\
\hline 
$\Gamma(i,l,z,I^*)$ & Subset of $\gamma \in \Gamma$, for which $Z(i,l,\gamma) = z$ and $g^*(d(\varepsilon,\gamma)) = A_{I^*}d(\varepsilon, \gamma)$.\\
\hline
$Y(i,l,z,I^*)$ & Subset of realizations of $Y$ having ratios between its components given by some $\gamma\in \Gamma(i,l,z,I^*)$.\\
                    \hline

    \end{tabular}}
    \vspace{0.2cm}
    \caption{Summary of notation.}
    \label{tab:my_label}
\end{table}}
\section{Main results} \label{mainResults}
In this section, we introduce and prove our main result, Theorem \ref{mainThm}. First, in Section \ref{roadmap}, we state the theorem, provide all results required for its proof, and sketch the steps of the proof, which we fully give in Section \ref{proofThm}. The proofs of the remaining results, unless stated otherwise, are provided in Section \ref{proofs}.
\subsection{Main theorem and roadmap of the proof} \label{roadmap}
Our main result, Theorem \ref{mainThm}, shows that in the case of graph disconnection, the total failure size has a Pareto tail with the same parameter as the demand distribution. More specifically, if under demand profile \eqref{genericD} there is a positive probability of network disconnection, then the total failure size $S$ is Pareto-tailed with parameter $\alpha$. Otherwise, $S$ has a lighter tail. 

\begin{theorem}[Tail of the total cascade size] \label{mainThm}
Let $\mathcal{G} = (\mathcal{N},\mathcal{E})$ and $\lambda\in (0,1)$. If there exist a node $i\in \mathcal{N}$, a first edge failure $l\in \mathcal{E}$, a hyperplane index set $I^\ast\in \mathcal{I}$, and a number of nodes $z\in \mathcal{Z}/\{0\}$ disconnected from node $i$ such that $\mathbb{P}\left(Y\in Y(i,l,z,I^\ast)\right)>0$, then \[\mathbb{P}(S>x)\sim Cx^{-\alpha} \text{ as } x\rightarrow \infty,\]
where  $$C := \sum_{i\in \mathcal{N}}\sum_{l\in\mathcal{E}}\sum_{I^\ast\in\mathcal{I} }\sum_{z\in \mathcal{Z}}\frac{K}{m}\left(\frac{\lambda z}{n}\right)^{\alpha}\mathbb{P}\left( Y\in Y(i,l,z,I^\ast)~|~ N = i, L^{(1)} = l\right)$$
and $K$ is the normalizing constant of a Pareto-tailed random variable as given in \eqref{Kconstant}.
Otherwise, $\mathbb{P}(S>x) = \mathcal{O}(x^{-2\alpha})$.
\end{theorem}

 The theorem is a generalization of \cite[Theorem IV.12]{Nesti_2020}, as it does not require additional assumptions on $\lambda$ and $\mathcal{G}$. In Section \ref{emergency}, we show that these assumptions are not merely technical, as they exclude certain graphs. In particular, in Example \ref{example}, we demonstrate that the graph in Figure \ref{fig:directed-graph} does not meet the assumptions of the theorem in \cite{Nesti_2020} because in the 4-th step of the cascade, three edges maximize the relative flow exceedance. With our theorem, the analysis of the cascade size in this graph becomes feasible.

 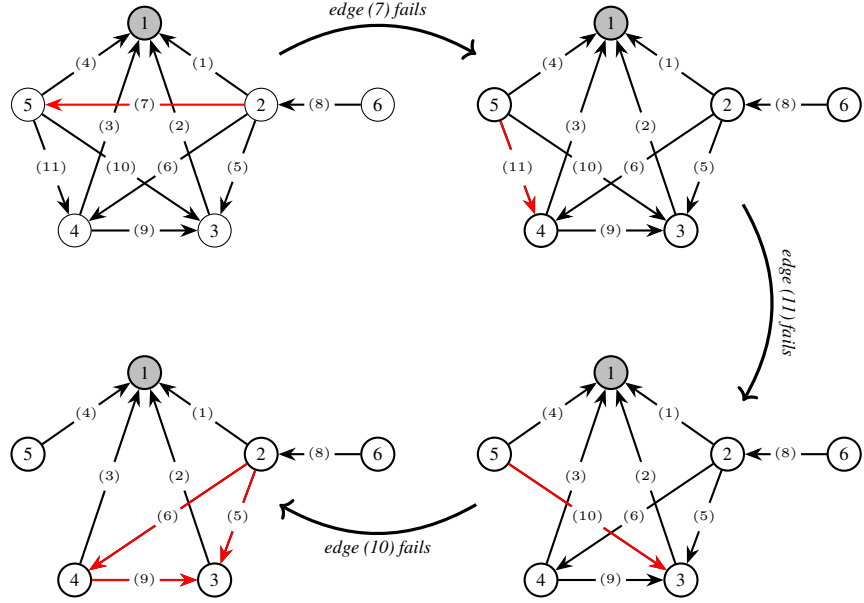
\begin{figure}[h]
\centering
\begin{tikzpicture}[ scale=0.31, every node/.style={scale=0.7}]
    \node (b) at (14.5,-2.6) {};
    \node (a) at (5.5,-2.6) {};
    \node(c) at (14.5,-21.6) {};
    \node (d) at (5.5,-21.6) {};
    \node (e) at (25.5,-8.6) {};
    \node (f) at (25.5,-17.6) {};
    \path[<-, very thick] (b) edge[bend right,very thick] node[above] {\textit{edge (7) fails}}(a);
    \path[->, very thick] (e) edge[bend left] node[above, rotate=-90] {\textit{edge (11) fails}} (f);
    \path[->, very thick] (c) edge[bend left] node[below] {\textit{edge (10) fails}}(d);
        \node[circle,thick,draw, fill = lightgray] (1) at (0,-1.1) {1};
\begin{scope}[every node/.style={circle,draw,scale=0.7}]

    \node (2) at (5,-4.6) {2};
    \node (3) at (3,-10) {3};
    \node (4) at (-3,-10) {4};
    \node (5) at (-5,-4.6) {5};
    \node (6) at (10,-4.6) {6};
\end{scope}

\begin{scope}[>={Stealth[black]},
              every node/.style={fill=white,circle,inner sep=0.5pt, font=\fontsize{4.5}{0}},
              every edge/.style={draw=black,thick, scale=0.7}]
    \path [<-] (1) edge node {$(1)$} (2);
    \path [<-] (1) edge node {$(2)$} (3);
    \path [<-] (1) edge node {$(3)$} (4);
    \path [<-] (1) edge node {$(4)$} (5);
    \path [->] (2) edge node {$(5)$} (3);
    \path [->] (2) edge node {$(6)$} (4);
    \path [<-] (2) edge node {$(8)$} (6);
    \path [->] (4) edge node {$(9)$} (3);
    \path [<-] (3) edge node {$(10)$} (5);
    \path [<-] (4) edge node {$(11)$} (5);
\end{scope}
\begin{scope}[>={Stealth[red]},
              every node/.style={fill=white,circle,inner sep=0.5pt, font=\fontsize{4.5}{0}},
              every edge/.style={draw = red,thick}]
    \path [<-] (5) edge node {$(7)$} (2);
\end{scope}
\begin{scope}[every node/.style={circle,thick,draw, scale = 0.7}]
    \node[circle,thick,draw, fill = lightgray] (7) at (20,-1.1) {1};
    \node (8) at (25,-4.6) {2};
    \node (9) at (23,-10) {3};
    \node (10) at (17,-10) {4};
    \node (11) at (15,-4.6) {5};
    \node (12) at (30,-4.6) {6};
\end{scope}

\begin{scope}[>={Stealth[black]},
              every node/.style={fill=white,circle,inner sep=0.5pt, font=\fontsize{4.5}{0}},
              every edge/.style={draw=black,thick}]
    \path [<-] (7) edge node {$(1)$} (8);
    \path [<-] (7) edge node {$(2)$} (9);
    \path [<-] (7) edge node {$(3)$} (10);
    \path [<-] (7) edge node {$(4)$} (11);
    \path [->] (8) edge node {$(5)$} (9);
    \path [->] (8) edge node {$(6)$} (10);
    \path [<-] (8) edge node {$(8)$} (12);
    \path [->] (10) edge node {$(9)$} (9);
    \path [<-] (9) edge node {$(10)$} (11);
    \path [<-] (10) edge node {$(11)$} (11);
\end{scope}

\begin{scope}[>={Stealth[red]},
              every node/.style={fill=white,circle,inner sep=0.5pt, font=\fontsize{4.5}{0}},
              every edge/.style={draw = red,thick}]
    \path [<-] (10) edge node {$(11)$} (11);
\end{scope}

\begin{scope}[every node/.style={circle,thick,draw, scale = 0.7}]
    \node (1d)[circle,thick,draw, fill = lightgray] at (0,-16.1) {1};
    \node (2d) at (5,-19.6) {2};
    \node (3d) at (3,-25) {3};
    \node (4d) at (-3,-25) {4};
    \node (5d) at (-5,-19.6) {5};
    \node (6d) at (10,-19.6) {6};
\end{scope}

\begin{scope}[>={Stealth[black]},
              every node/.style={fill=white,circle,inner sep=0.5pt, font=\fontsize{4.5}{0}},
              every edge/.style={draw=black,thick}]
    \path [<-] (1d) edge node {$(1)$} (2d);
    \path [<-] (1d) edge node {$(2)$} (3d);
    \path [<-] (1d) edge node {$(3)$} (4d);
    \path [<-] (1d) edge node {$(4)$} (5d);
    \path [->] (2d) edge node {$(5)$} (3d);
    \path [->] (2d) edge node {$(6)$} (4d);
    \path [<-] (2d) edge node {$(8)$} (6d);
    \path [->] (4d) edge node {$(9)$} (3d);
\end{scope}

\begin{scope}[>={Stealth[red]},
              every node/.style={fill=white,circle,inner sep=0.5pt, font=\fontsize{4.5}{0}},
              every edge/.style={draw=red,thick}]
    \path [->] (2d) edge node {$(5)$} (3d);
    \path [->] (2d) edge node {$(6)$} (4d);
    \path [->] (4d) edge node {$(9)$} (3d);
\end{scope}

\begin{scope}[every node/.style={circle,thick,draw, scale = 0.7}]
    \node (1c)[circle,thick,draw, fill = lightgray] at (20,-16.1) {1};
    \node (2c) at (25,-19.6) {2};
    \node (3c) at (23,-25) {3};
    \node (4c) at (17,-25) {4};
    \node (5c) at (15,-19.6) {5};
    \node (6c) at (30,-19.6) {6};
\end{scope}

\begin{scope}[>={Stealth[black]},
              every node/.style={fill=white,circle,inner sep=0.5pt, font=\fontsize{4.5}{0}},
              every edge/.style={draw=black,thick}]
    \path [<-] (1c) edge node {$(1)$} (2c);
    \path [<-] (1c) edge node {$(2)$} (3c);
    \path [<-] (1c) edge node {$(3)$} (4c);
    \path [<-] (1c) edge node {$(4)$} (5c);
    \path [->] (2c) edge node {$(5)$} (3c);
    \path [->] (2c) edge node {$(6)$} (4c);
    \path [<-] (2c) edge node {$(8)$} (6c);
    \path [->] (4c) edge node {$(9)$} (3c);
    \path [<-] (3c) edge node {$(10)$} (5c);
\end{scope}

\begin{scope}[>={Stealth[red]},
              every node/.style={fill=white,circle,inner sep=0.5pt, font=\fontsize{4.5}{0}},
              every edge/.style={draw=red,thick}]
    \path [<-] (3c) edge node {$(10)$} (5c);
\end{scope}

\end{tikzpicture}

\caption{Example of a graph with a positive probability of a non-unique cascade sequence. Specifically, if node 1 has the largest demand, then the failure of edge (7) leads to the subsequent failure of edges (11) and (10). At this stage of the cascade, edges (5), (6), and (9) are the maximizers of the flow exceedance.}
\label{fig:directed-graph}
\end{figure}

To be able to prove the theorem, we first take the following steps:
\begin{enumerate}
    \item Motivated by Lemma \ref{BigJumpLemma}, we show that it is sufficient to consider demand profiles as given in \eqref{genericD}.
    \item In Lemma \ref{AltOptSol}, we reformulate the optimal generation in terms of a projection onto the polytope $\mathcal{F}_d$.
    \item In Proposition \ref{propProjection}, we decompose the optimal generation into the matrix equation $g^*(d) = A_{I^*(d)}d$.
    \item In Lemma \ref{firstColumnAgeneral}, we show that for all demand vectors $d(\varepsilon,\gamma)$ in the neighborhood of $\varepsilon = 0$, the first column of the matrix $A_{I^*(d)}$ is fixed and equal to $(1-\lambda)e_1 + \frac{\lambda}{n}e.$ We exploit this property in the proofs of Proposition \ref{propositionKKT} and Theorem \ref{mainThm}.
    \item In Proposition \ref{propositionKKT}, we show that $A_{I^*(d)}$ depends on $d(\varepsilon,\gamma)$ only through $\gamma$ in the neighborhood of $\varepsilon = 0$. 
    \item Using the result of Lemma \ref{cascadeEpsilon}, we partition $\Gamma$ into finitely many subsets, each corresponding to a particular subset $I^*\in \mathcal{I}$ and a particular cascade sequence in the neighborhood of $\varepsilon = 0$. This yields the corresponding partition of the feasible demand space into subsets $Y(i,l,z,I^*)$.
    \item In Lemma \ref{asymptoticS}, we derive the asymptotic behavior of the probability that $Y_1>x$, $\sum_{j = 2}^n Y_j\leq \varepsilon Y_1$, and $Y\in Y(i,l,z,I^*)$, given that the node with the largest demand has label $i$ and edge $l$ fails first, as $x\rightarrow \infty$. This is a crucial ingredient in the proof of the main result.
\end{enumerate}

Steps 1 to 6 carefully partition the feasible demand space into subsets with certain generation and cascade characteristics. Using this insight, we prove Theorem \ref{mainThm} in the following way. We first divide the demand space into two sets, one of which induces only light-tailed cascade sizes. For the remaining set, we apply the previously derived partition to ensure certain cascade characteristics within each subset. Then, we construct upper and lower bounds for the total cascade size $S$. Using the result of Step 7, we prove that both bounds are asymptotically equal and yield heavy-tailed behavior. This proof strategy is executed, using the following steps.
\begin{enumerate}[resume]
    \item We condition on the label of the node with the largest demand and the first edge failure. 
    \item Motivated by Lemma \ref{BigJumpLemma}, we consider the two cases $ \sum_{j = 2}^n Y_j\leq \varepsilon Y_1$ and $ \sum_{j = 2}^n Y_j> \varepsilon Y_1$. For the latter case, the probability of the total cascade size is of order $\mathcal{O}(x^{-2\alpha})$.
    \item In the former case, we study the total cascade for each event $\{Y\in Y(i,l,z,I^*)\}$ that occurs with positive probability. For $Y\in Y(i,l,z,I^*)$ and $ \sum_{j = 2}^n Y_j\leq \varepsilon Y_1$, the optimal generation and the cascade sequence are known due to Proposition $\ref{propositionKKT}$ and Lemma \ref{cascadeEpsilon}. If none of the events induce a positive total cascade size, we conclude that the tail probability of $S$ is of order $\mathcal{O}(x^{-2\alpha})$, as given in Step 9. Otherwise, we provide an upper and lower bound for the total cascade size $S$ as functions of $i,l,z,I^*$, and $\varepsilon$. These bounds scale linearly with the largest demand. Importantly, as $\varepsilon \downarrow 0$, both bounds converge to the same limit, which is known due to Lemma~\ref{firstColumnAgeneral}.
    \item We apply the results from Steps 7 and 10 to obtain asymptotic expressions for the upper and lower bound for the probability of the total cascade size $S$, showing that both depend on $\varepsilon$.
    \item We take $\varepsilon\downarrow 0$ for both bounds derived in Step 11. The limits are equal and yield an asymptotic expression for the probability of the total cascade size with $Y\in Y(i,l,z,I^*)$ and $\sum_{j = 2}Y_j\leq \varepsilon Y_1$. Summing up over all possible large demand nodes $i\in \mathcal{N}$, first edge failures $l\in \mathcal{E}$, number of disconnected nodes $z\in \mathcal{Z}$ and hyperplane index sets $I^*\in \mathcal{I}$ yields the result of the theorem.
\end{enumerate}

Our proof bears similarities with the proof of Theorem IV.12 in \cite{Nesti_2020}. However, major differences occur in Steps 10 and 11. In Step 10, since we do not impose additional assumptions, the continuity argument used in \cite{Nesti_2020} is not valid. To circumvent this problem, we partition all feasible demands into subsets with properties that allow us to derive the expressions for the total cascade size in each scenario. Furthermore, the proof of \cite[Theorem IV.12]{Nesti_2020} relies on existing results for the asymptotic behavior of $\mathbb{P}(Y_1>x, \sum_{j = 2}^n Y_j\leq \varepsilon Y_1\}$ as $x\rightarrow \infty$. Our proof requires an analogous asymptotic result, derived in Lemma \ref{asymptoticS}, which additionally takes into account the partition of the demand space. We now proceed to describe Steps 1 up to 7 in more detail, provide formal results, and formally prove Theorem \ref{mainThm}. 

In the first step, we rely on the following lemma, proven in the existing literature. This lemma implies that the dominant scenario for a large total cascade involves one node possessing a significantly greater demand than all other nodes combined. Therefore, we can confine our analysis to the demands specified in \eqref{genericD}; all other scenarios can be disregarded in the asymptotic regime. 

\begin{lemma}[Lemma IV.1 of \cite{Nesti_2020}] \label{BigJumpLemma}
    Suppose $X_i$, $i = 1,\dots, n$, are independent and identically distributed Pareto-tailed demands with parameter $\alpha >0$, i.e. $\mathbb{P}(X_i>x) \sim Kx^{-\alpha}$ as $x\rightarrow \infty$. Let $S$ be the total cascade failure size. Then, for every $\varepsilon>0,$ as $x\rightarrow \infty $, 
    \[\mathbb{P}\left(S>x; \sum_{j = 1, j\neq i}^nX_j\geq \varepsilon X_i\right) = \mathcal{O}\left(x^{-2\alpha}\right),\]
    where $i:= \argmax_{i = 1,\dots,n}\{X_i\}$.\end{lemma}

In the second step, we give an alternative formulation of the DC-OPF program \eqref{DC-OPF} that provides a geometrical interpretation of the optimal generation. Specifically, $g^*(d)$ is a point in $\mathcal{F}_d$ minimizing the Euclidean distance to $\bar{d}e$. In other words, $g^*(d)$ can be constructed as an orthogonal projection. This is formally stated in the following lemma.
\begin{lemma}
The vector $g^\ast(d)$ is an optimal solution of the problem
\[\min_{g\in \mathcal{F}_d} \left\{||g - \bar{d}e||_2\right\}\]
if and only if it is an optimal solution of \eqref{DC-OPF}.
    \label{AltOptSol}
\end{lemma}

In the third step, we leverage the theory of orthogonal projections to establish that for a given $d$, there exists a matrix~$A_{I^*(d)}$, such that $g^\ast(d) = A_{I^*(d)}d$. Furthermore, we characterize the matrix $A_{I^*(d)}$ in terms of the hyperplane index set $I^*(d)$, which defines the face on which $g^*(d)$ lies. There are finitely many subsets $I^*(d)$ because the polytope $\mathcal{F}_d$ has finitely many faces. Hence, the set of solutions $g^*(d)$ can be fully characterized by using a finite number of matrices. This result is given in the following proposition.
\begin{proposition}\label{propProjection}
Let $d$ be a demand vector satisfying \eqref{genericD} with $\varepsilon\geq 0$ sufficiently small. Then,
\hspace{1cm}
    \begin{enumerate}
        \item 
 there exists a matrix $A_{I^*(d)}\in \mathbb{R}^{n\times n}$ such that 
\[g^*(d) = A_{I^*(d)}d;\]
\item there exists a \textit{finite} set of matrices $\mathcal{A}$ such that $g^*(d) = A_{I^*(d)} d$ for some $A_{I^*(d)}\in \mathcal{A}$\,. 
    \end{enumerate}
\end{proposition}
In the fourth step, we derive the following lemma, stating that the first column of $A_{I^*(d)}$ is fixed for all $\gamma\in \Gamma$ and all $\varepsilon>0$ sufficiently small. This column is equal to the optimal generation induced by demand vector $e_1$, i.e., the demand vector in the limiting case of $\varepsilon = 0$.
\begin{lemma} \label{firstColumnAgeneral}
Let $I^*(d)$ be the hyperplane index set for a demand vector $d$ satisfying \eqref{genericD} with $\varepsilon\geq 0$ sufficiently small. Then,
\[A_{I^*(d)}e_1 = (1-\lambda)e_1 + \frac{\lambda}{n}e.\]
\end{lemma}
In the fifth step, using the Karush–Kuhn–Tucker conditions, we show that for $\varepsilon$ sufficiently small, the matrix $A_{I^*(d)}$ depends on~$d$ only through $\gamma$. The subsequent proposition makes this statement explicit.

\begin{proposition} \label{propositionKKT}
    For a demand vector $d(\varepsilon,\gamma)$ defined in \eqref{genericD} and $\varepsilon>0$ small enough, the optimal generation $g^*(d)$ is given by
    \begin{equation}g^\ast(d) = A_{\gamma}d,\label{KKTequation}\end{equation}
for some $A_{\gamma}\in \mathbb{R}^{n\times n}$. More specifically, there exists a mapping $\phi: \Gamma\rightarrow \mathcal{I}$ such that the hyperplane index set is equal to $I^*(d) = \phi(\gamma)$ for all $\varepsilon>0$ sufficiently small. Hence, Equation \eqref{KKTequation} holds with $A_{\gamma} := A_{\phi(\gamma)}$.
\end{proposition}
In the sixth step, we derive the following lemma, which shows that the cascade sequence also depends on $d$ only through $\gamma$ in the neighborhood of $\varepsilon = 0$. 
\begin{lemma}\label{cascadeEpsilon}
Assume that node $i\in \mathcal{N}$ has the largest demand and edge $e\in \mathcal{E}$ is the first to fail. Then, for a given $\gamma\in \Gamma$, the cascade sequence and the resulting graph $\mathcal{G}^{(\text{end})}$ are independent of $\varepsilon$ for all $\varepsilon>0$ sufficiently small. \end{lemma}

As a consequence of Lemma \ref{cascadeEpsilon}, given the largest demand node $N = i$ and the first edge failure $L^{(1)} = l,$ the number of nodes disconnected from $i$ is uniquely defined by $\gamma$ in the neighborhood of $\varepsilon$. As a result, we can classify the demands into subsets $Y(i,l,z,I^*)$, as defined in \eqref{demandsPartition}. Each demand in a subset has the following properties: first, it has its largest component at node $i$; second, it yields generation characterized by the hyperplane index set $I^*$; and last, it results in $z$ nodes disconnected from node $i$ after the cascade initiated by edge $l$.

In the seventh step, we derive an asymptotic result for $\mathbb{P}\left(Y_1>x, \sum_{j = 2}^nY_j\leq \varepsilon Y_1, Y\in Y(i,l,z,I^*)\right)$ as $x\rightarrow~\infty$. Our result shows that, in the limit, the event $\{Y\in Y(i,l,z,I^*)\}$ becomes independent of the event \mbox{\(\{Y_1>x,\sum_{j = 2}^nY_j\leq\varepsilon Y_1\}\)}. Moreover, $\mathbb{P} (Y_1>x,\sum_{j = 2}^n Y_j\leq \varepsilon Y_1)\sim n\mathbb{P}(X_1>x)$. This is an extension of the \textit{principle of a single big jump} that states that $\mathbb{P}(Y_1>x)\sim n\mathbb{P}(X_1>x)$ \cite{nair_wierman_zwart_2022}. Our result is given in the following lemma.
\begin{lemma} \label{asymptoticS}
Consider $X_1,\dots, X_n$ independent and identically distributed Pareto random variables with parameter $\alpha>0$, and the corresponding random variable $Y$ as defined in \eqref{Y}. Then, for $\varepsilon>0$ sufficiently small and all $i\in \mathcal{N}$, $l\in \mathcal{E}$, $z\in \mathcal{Z}$, and $I^\ast\in \mathcal{I}$, for which $\mathbb{P}(Y\in Y(i,l,z,I^*))>0$, we have that
\[\mathbb{P}\left(Y_1>x, \hspace{-0.1cm}\sum_{j = 2}^n Y_j\leq \varepsilon Y_1,Y\in Y(i,l,z, I^\ast)| N = i, L^{(1)} = l\right)\sim n\mathbb{P}(X_1>x)\mathbb{P}(Y\in Y(i,l,z, I^\ast)| N = i, L^{(1)} = l).\]
\end{lemma} 
Using the above results, in the following section we provide the proof of the main theorem.
\subsection{Proof of Theorem \ref{mainThm}} \label{proofThm}
In this section, we prove the main result, following Steps 8 -- 12.
\begin{proof}
First, following Step 8, we observe that the cascade sequence depends the node $N$ that has the largest demand as well as on which edge failed first, denoted by $L^{(1)}$. Therefore, we condition on these events. Recall that $\mathbb{P}(N = i) = 1/n$ for all $i$ and $\mathbb{P}(L^{(1)} = l) = 1/m$ for all $l\in \mathcal{E}$. Thus, we obtain
\[\mathbb{P}(S>x) = \sum_{i\in \mathcal{N}}\sum_{l\in \mathcal{E}}\frac{1}{nm}\mathbb{P}(S>x|N = i, L^{(1)} = l).\] 
Next, let $\varepsilon$ be sufficiently small so that $Z(i,l,\gamma)$ is well-defined. This can be done due to Lemma \ref{cascadeEpsilon}, which shows that the cascade sequence is independent of $\varepsilon$ in this regime. Then, we decompose the tail probability of the conditional total failure size in two cases:
\begin{align*}\mathbb{P}(S>x|N = i,L^{(1)} = l) &= \mathbb{P}(S>x, \sum_{j = 2}^n Y_j< \varepsilon Y_1|N = i,L^{(1)} = l) + \mathbb{P}(S>x, \sum_{j = 2}^n Y_j\geq  \varepsilon Y_1|N = i,L^{(1)} = l).\end{align*}
From Lemma \ref{BigJumpLemma}, it follows that $\mathbb{P}(S>x, \sum_{j=2}^n Y_j\geq  \varepsilon Y_1|N = i,L^{(1)} = l) = \mathcal{O}(x^{-2\alpha})$. Thus, only the first term requires further consideration. This marks the end of Step 9. 

Next, let $Z(i,l)$ be the expected number of nodes disconnected from node $i$ after a cascade sequence that was initiated by the failure of edge $l$ in the neighborhood of $\varepsilon = 0$, i.e.
\begin{align*}Z(i,l) &:= 
\sum_{I^\ast\in\mathcal{I}}\sum_{z\in \mathcal{Z}}z\cdot \mathbb{P}( Y\in Y(i,l,z,I^\ast) ~|~ N = i, L^{(1)} = l).\end{align*}

Suppose that for all nodes $i$, edges $l$, $z\in \mathcal{Z}/\{0\}$, and $I^\ast\in \mathcal{I}$, we have $\mathbb{P}(Y\in Y(i,l,z,I^\ast) | N = i, L^{(1)} = l) = 0$. This implies that for all $i\in \mathcal{N}$ and $l\in \mathcal{E}$, $Z(i,l) = 0$. Consequently, almost surely, no node has been disconnected from node $i$ for $\varepsilon>0$ sufficiently small. Thus, the cascade size in this setting is equal to 0 with probability 1. In particular, for $x>0$,
\[\mathbb{P}(S>x, \sum_{j = 2}^n Y_j< \varepsilon Y_1|N = i,L^{(1)} = l)  = 0.\]
Altogether, if $\mathbb{P}(Y\in Y(i,l,z,I^\ast) | N = i, L^{(1)} = l) = 0$ for all $i\in \mathcal{N}$, $l\in \mathcal{E}$, $z\in \mathcal{Z}/\{0\}$, and $I^\ast\in \mathcal{I}$, then the cascade size satisfies
\[\mathbb{P}(S>x) =\sum_{i\in \mathcal{N}}\sum_{l\in \mathcal{E}}\frac{1}{nm}\left(0 + \mathcal{O}(x^{-2\alpha})\right) =  \mathcal{O}(x^{-2\alpha}),\]
as argued in Step 10. Hence, from now on, we assume that there exist $i\in \mathcal{N}$, $l\in \mathcal{E}$,  $z\in \mathcal{Z}/\{0\}$, and $I^\ast\in \mathcal{I}$ such that \mbox{$\mathbb{P}(Y\in Y(i,l,z,I^\ast) | N = i, L^{(1)} = l) > 0$}. 

It is known that the solution of the DC-OPF program \eqref{DC-OPF} is scale-invariant, i.e. $g^*(\beta Y) = \beta g^*(Y)$, for any demand vector $Y$ and $\beta> 0$ \cite{Nesti_2020}. Moreover, we know that the normalized demand vector $Y/Y_1$ is equal to some demand profile $d(\varepsilon,\gamma)$. Hence, due to the scale-invariance of $g^*(d)$ and Proposition \ref{propositionKKT}, it holds that
\[g^*(Y) = Y_1g^*(Y/Y_1) = Y_1g^*(d(\varepsilon,\gamma)) = Y_1A_{I^*(d)}d(\varepsilon,\gamma) = A_{I^*(d)}Y.\] This means that the linear form of the optimal generation holds not only for the normalized demand profiles but also more generally for $Y$, which we exploit from now on. 
Notably, the cascade corresponding to demand $Y$ is the same as the cascade corresponding to $d(\varepsilon, \gamma)$. This is due to the linearity of the flow vector $f$, see \eqref{flow}, and the emergency edge limit vector $F$, see \eqref{emergencyLimit}, with respect to demand $d$ and generation $g$. This property implies that the relative exceedance vector $\psi$, see \eqref{psi}, depends solely on the proportions between the elements of $d$, i.e. the vector $\gamma$, and not the scale factor $Y_1$. 

Let, $A(i,l,z,I^\ast)$ denote the set of nodes connected to the node with the largest demand after the cascade with demand $Y\in Y(i,l,z,I^*)$. Note that due to the previous argument, $A(i,l,z,I^*)$ is well-defined. Again, using the law of total probability and the fact that we know the solution form of $g^*(Y)$ given in Proposition \ref{propositionKKT}, we find
\begin{align*}
    &\mathbb{P}(S>x,\sum_{j = 2}^nY_j<\varepsilon Y_1 | N = i, L^{(1)} = l) \\
    &\hspace{0.6cm}= \sum_{I^\ast\in \mathcal{I}}\sum_{z\in \mathcal{Z}}\mathbb{P}(S>x,\sum_{j = 2}^nY_j<\varepsilon Y_1, Y\in Y(i,l,z,I^\ast) | N = i, L^{(1)} = l) \\
    &\hspace{0.6cm}=\sum_{I^\ast\in \mathcal{I}}\sum_{z\in \mathcal{Z}}\mathbb{P}(\sum_{j\not\in A(i,l,z,I^\ast)} g^*_j - Y_j>x,\sum_{j = 2}^nY_j<\varepsilon Y_1, Y\in Y(i,l,z,I^\ast) | N = i, L^{(1)} = l)\\
    &\hspace{0.6cm}= \sum_{I^\ast\in \mathcal{I}}\sum_{z\in \mathcal{Z}}\mathbb{P}(\sum_{j\not\in A(i,l,z,I^\ast)} e_j^T(A_{I^\ast} - I)Y>x, \sum_{j = 2}^nY_j<\varepsilon Y_1, Y\in Y(i,l,z,I^\ast)|N = i, L^{(1)} = l).
\end{align*}
Note that the conditions $\sum_{j = 2}^nY_j<\varepsilon Y_1$ and $Y\in Y(i,l,z,I^\ast)$ together imply that $Y = (e_1 + \hat{\varepsilon} \gamma )Y_1$ for some $\hat{\varepsilon}<\varepsilon$ and $\gamma\in \Gamma(i,l,z,I^\ast)$. 
Let $j\not\in A(i,l,z,I^\ast)$, which implies that $j\neq 1$. Then, using Lemma \ref{firstColumnAgeneral}, we obtain
\[e_j^T(A_{I^\ast} - I)e_1=\frac{\lambda}{n}e_j^Te -\lambda e_j^Te_1 = \frac{\lambda}{n}. \]
Furthermore, let $B(i,l,z,I^\ast):= \max_{j\not\in A(i,l,z,I^\ast), \gamma\in \Gamma(i,l,z,I^\ast)}\{|e_j^T(A_{I^\ast} - I)\gamma)|\}$. Thus, we find the following bounds:
\[z\left(\frac{\lambda}{n} - B(i,l,z,I^\ast)\varepsilon\right)\leq \sum_{j\not\in A(i,l,z,I^\ast)}e^T_j(A_{I^\ast} - I)(e_1 + \hat{\varepsilon}\gamma)\leq z\left(\frac{\lambda}{n} + B(i,l,z,I^\ast)\varepsilon\right).\]
Using the above bounds, we find
\begin{multline*}
    \mathbb{P} \Big(S>x,\sum_{j = 2}^nY_j<\varepsilon Y_1 | N = i, L^{(1)} = l \Big) \\\leq  \sum_{I^\ast\in \mathcal{I}}\sum_{z\in \mathcal{Z}}\mathbb{P}\left(z\left(\frac{\lambda}{n} + B(i,l,z,I^\ast)\varepsilon\right)Y_1>x, \sum_{j = 2}^nY_j<\varepsilon Y_1, Y\in Y(i,l,z,I^\ast)|N = i, L^{(1)} = l\right).
\end{multline*}
Similarly, we can derive a lower bound:
\begin{multline*}
    \mathbb{P}\Big(S>x,\sum_{j = 2}^nY_j<\varepsilon Y_1 | N = i, L^{(1)} = l\Big) \\\geq \sum_{I^\ast\in \mathcal{I}}\sum_{z\in \mathcal{Z}}\mathbb{P}\left(z\left(\frac{\lambda}{n} - B(i,l,z,I^\ast)\varepsilon\right)Y_1>x, \sum_{j = 2}^nY_j<\varepsilon Y_1, Y\in Y(i,l,z,I^\ast)|N = i, L^{(1)} = l\right).
\end{multline*}
We now move on to Step 11, where we derive the asymptotic behavior of the above bounds. Applying the result of Lemma \ref{asymptoticS} to all summands, for which $\mathbb{P}(Y\in Y(i,l,z,I^*))>0$, we find that
\begin{align*}&\lim_{x\rightarrow \infty}\frac{\mathbb{P}(S>x,\sum_{j = 2}^nY_j<\varepsilon Y_1 | N = i, L^{(1)} = l)}{n\mathbb{P}(X_1>x)}\\
&\hspace{0.6cm}=\lim_{x\rightarrow \infty}\sum_{I^\ast\in \mathcal{I}}\sum_{z\in \mathcal{Z}}\frac{\mathbb{P}(S>x,\sum_{j = 2}^nY_j<\varepsilon Y_1, Y\in Y(i,l,z,I^\ast) | N = i, L^{(1)} = l)}{n\mathbb{P}(X_1>x)}\\
&\hspace{0.6cm}=\sum_{I^\ast\in \mathcal{I}}\sum_{z\in \mathcal{Z}}\lim_{x\rightarrow \infty}\frac{\mathbb{P}(S>x,\sum_{j = 2}^nY_j<\varepsilon Y_1 , Y\in Y(i,l,z,I^\ast)| N = i, L^{(1)} = l)}{n\mathbb{P}(X_1>x)}\\
&\hspace{0.6cm}\leq\sum_{I^\ast\in \mathcal{I}}\sum_{z\in \mathcal{Z}} \lim_{x\rightarrow \infty}\frac{\mathbb{P}\left(z\left(\frac{\lambda}{n} + B(i,l,z,I^\ast)\varepsilon\right)Y_1>x, \sum_{j = 2}^nY_j<\varepsilon Y_1, Y\in Y(i,l,z,I^\ast)| N = i, L^{(1)} = l\right)}{n\mathbb{P}(X_1>x)}\\
&\hspace{0.6cm}= \sum_{I^\ast\in \mathcal{I}}\sum_{\substack{z\in \mathcal{Z}:\\\mathbb{P}(Y\in Y(i,l,z,I^*)>0}} K\left(z\left(\frac{\lambda}{n} + B(i,l,z,I^\ast)\varepsilon\right)\right)^{\alpha}\cdot \mathbb{P}(Y\in Y(i,l,z,I^\ast)| N = i, L^{(1)} = l)\\
&\hspace{0.6cm}= \sum_{I^\ast\in \mathcal{I}}\sum_{z\in \mathcal{Z}} K\left(z\left(\frac{\lambda}{n} + B(i,l,z,I^\ast)\varepsilon\right)\right)^{\alpha}\cdot \mathbb{P}(Y\in Y(i,l,z,I^\ast)| N = i, L^{(1)} = l), \end{align*}
for some $K>0$. Similarly, we find a lower bound:
\begin{multline*}\lim_{x\rightarrow \infty}\frac{\mathbb{P}(S>x,\sum_{j = 2}^nY_j<\varepsilon Y_1 | N = i, L^{(1)} = l)}{n\mathbb{P}(X_1>x)}\\
\geq \sum_{I^\ast\in \mathcal{I}}\sum_{z\in \mathcal{Z}} K\left(z\left(\frac{\lambda}{n} - B(i,l,z,I^\ast)\varepsilon\right)\right)^{\alpha}\cdot \mathbb{P}(Y\in Y(i,l,z,I^\ast)| N = i, L^{(1)} = l). \end{multline*}
In the final step, taking the limit $\varepsilon\rightarrow 0$, we obtain the equality of both bounds. Hence, we conclude that, as $x\rightarrow \infty$,
\begin{multline*}\mathbb{P}(S>x) \sim \sum_{i\in \mathcal{N}}\sum_{l\in \mathcal{E}}\sum_{I^\ast\in \mathcal{I}}\sum_{z\in \mathcal{Z}} \frac{K}{m}\left(\frac{\lambda z}{n}\right)^\alpha \cdot \mathbb{P}(Y\in Y(i,l,z,I^\ast) | N = i, L^{(1)} = l) \cdot  x^{-\alpha}\\
= \sum_{i\in \mathcal{N}}\sum_{l\in \mathcal{E}}\frac{K}{m}\left(\frac{\lambda}{n}\right)^{\alpha}\sum_{I^\ast\in \mathcal{I}}\sum_{z\in \mathcal{Z}} z^\alpha \cdot \mathbb{P}(Y\in Y(i,l,z,I^\ast) | N = i, L^{(1)} = l) \cdot  x^{-\alpha}.\end{multline*}
Note that since we assume that $\mathbb{P}(Y\in Y(i,l,z,I^\ast)|N = i,L^{(1)} = l)>0$ for some $i,l,I^\ast$ and $z>0$, the constant in front of $x^{-\alpha}$ is strictly positive. Thus, the result of the theorem follows.
\end{proof}
Theorem \ref{mainThm} shows us that the total failure size exhibits a Pareto tail with parameter $\alpha$, which is the same parameter that governs the Pareto tail of the demand distribution. This finding is in line with the previous results that show this property for a particular subset of graphs. In fact, Lemma \ref{equivThm} in Section \ref{impact} shows that our theorem, when applied to a graph that meets the conditions of Theorem IV.12 of \cite{Nesti_2020}, yields the exact same result.
\section{Proofs and supporting results}
\label{proofs}

In this section, we prove the previously introduced lemmas and propositions. Some of the proofs rely on auxiliary results, which are introduced and proved when appropriate. 

\subsection{Proof of Lemma \ref{AltOptSol}}
To prove this lemma, we rewrite the generation cost function $\eqref{obj}$, in terms of a squared distance between $g$ and $\bar{d}e$.
\begin{proof}
Observe that for all $g \in \mathcal{F}_d$,
\[||g - \bar{d}e||^2_2 =\sum_{i = 1}^n (g_i \bar{d})^2 =
\sum_{i = 1}^ng_i^2 - 2\bar{d}\sum_{i = 1}^n g_i +
\sum_{i = 1}^n \bar{d}^2 =\sum_{i = 1}^n
g_i^2 - n\bar{d}^2,\]
where the final equality follows since $e^Tg = e^Td$ for all $g \in \mathcal{F}_d$. Since $\bar{d}$ is a constant, and taking the root is a concave function, we conclude that the minimum of $||g -\bar{d}e||_2$ is attained at the same point as the minimum of $\frac{1}{2}g^Tg$ over the feasible area $\mathcal{F}_d$.
\end{proof}
\subsection{Proof of Proposition \ref{propProjection}}
We prove Proposition \ref{propProjection}, using insights from Lemma \ref{AltOptSol}. This is done in two steps; first, we state Lemma \ref{lemmaProjection}, which provides an explicit formula for $A_{I^*(d)}$ using projection theory. Based on this result, we then prove the proposition.

\begin{lemma}\label{lemmaProjection}
For a demand vector $d$ satisfying \eqref{genericD}, let $I^\ast(d)$ be defined as in \eqref{Istar}. Consider the following partition of $I^*(d)$:
$$I^\ast(d) = \{0\}\cup I_1\cup I_2,~~ \text{such that}~~ I_1\subseteq\{1,\dots,m\}, I_2 \subseteq\{m+1,\dots, 2m\},$$
with $s_1:= |I_1|, s_2:= |I_2|$, and define 
$$V_{I_1}:= \begin{bmatrix} (V)_{j_1}\\
\vdots\\
(V)_{j_{s_1}}\end{bmatrix}, \quad V_{I_2}:= \begin{bmatrix} (V)_{k_1}\\
\vdots\\
(V)_{k_{s_2}}\end{bmatrix}, \quad I_1 = \{j_1,\dots,j_{s_1}\}, \quad I_2 = \{k_1,\dots,k_{s_2}\}.$$
Then, the optimal solution $g^*(d)$ is given by 
\begin{equation*}
    g^*(d) = A_{I^*(d)}d,
\end{equation*}
where $A_{I^*(d)} := \frac{1}{n}J + \begin{bmatrix} V_{I_1}^T & V_{I_2}^T\end{bmatrix}\left(
\begin{bmatrix}V_{I_1}\\ V_{I_2}\end{bmatrix}
\begin{bmatrix}V^T_{I_1} & V^T_{I_2}\end{bmatrix}\right)^{+}\begin{bmatrix}
(1-\lambda)V_{I_1} \\
(1+\lambda)V_{I_2}\end{bmatrix}\in \mathbb{R}^{n\times n}$ and $J$ is the all-ones $n\times n$ matrix.

\label{projectionMatrix}
\end{lemma}
The proof of this lemma relies on the theory of orthogonal projections and is provided in Appendix \ref{proofLemmaProjection}. Next, using Lemma \ref{projectionMatrix}, we prove Proposition \ref{propProjection}.
\begin{proof}[Proof of Proposition \ref{propProjection}]
    The first statement of the proposition follows directly from Lemma \ref{lemmaProjection}. It remains to show the existence of a finite set $\mathcal{A}$. Let $\mathcal{A}:= \{A_{I^*(d)} | I^*(d) \subseteq \{0,\dots,2m\} \}$. Clearly, for each $g^*(d)$, there exists a matrix $A\in \mathcal{A}$, namely $A_{I^*(d)}$, such that $g^*(d) = A_{I^*(d)}d$. Matrix $A_{I^*(d)}$, as defined in Lemma \ref{lemmaProjection}, depends only on $I^*(d)\subseteq\{0,\dots,2m\}$. Since there are at most $2^{2m+1}$ subsets of $\{0,\dots, 2m\}$, the size of $\mathcal{A}$ is bounded by $2^{2m+1}$. This completes the proof.
\end{proof}
\subsection{Proof of Lemma \ref{firstColumnAgeneral}}
In this section, we prove Lemma \ref{firstColumnAgeneral}, which characterizes the first column of matrices $A_{I^*(d)}$ for certain demand vectors $d$. However, the proof requires three auxiliary results, Lemmas \ref{fullLine} -- \ref{firstColumnA}, which we introduce below.

In Lemma \ref{fullLine}, we show that the number of possible hyperplane index sets $I^*\in \mathcal{I}$ that lead to the optimal solution can be restricted. This arises from the fact that, for every optimal solution, the flows on edges adjacent to the node with the largest demand are equal to the operational edge limit $\lambda Vd$. Consequently, the inequality constraints for these edges are always tight. This notion is formalized in the following lemma.
\begin{lemma}\label{fullLine}
Let $d$ be a demand vector as defined in \eqref{genericD} with $\varepsilon>0$ sufficiently small. Then, for all edges $e\in \mathcal{E}_1$, the following holds:
\begin{equation}\big(Vg^\ast(d) \big)_e = (1-\lambda)(Vd)_e.\label{fullLines}\end{equation}
In other words, the flow on all edges adjacent to node 1 is equal to the operational limit.
\end{lemma}
We prove the lemma by constructing an alternative solution to the DC-OPF problem \eqref{DC-OPF}, leading to a contradiction. The proof is provided in Appendix~\ref{proofFullLine}.

Next, we show that constraint $(1+\lambda)(Vd)_i$, $i = 1,\dots,m$ can be tight only if $(Ve_1)_i = 0$. 
\begin{lemma} \label{tightConstraints}
    For all $\varepsilon >0$ sufficiently small, if $(Ve_1)_i>0$ for $i \in \{1,\dots,m\}$, then the constraint $(1+\lambda)(Vd)_i$ is not tight. Specifically,
    \[(Vg^\ast(d))_i < (1+\lambda) (Vd)_i.\]
\end{lemma}
The proof of the lemma relies on a continuity argument, for which the details can be found in Appendix~\ref{ProofTightConstraints}. Next, we show that the first row of matrices $A_{I^*}$ is known for certain choices of hyperplane index set $I^*\in \mathcal{I}$. 
\begin{lemma}\label{firstColumnA}
Let $I^* \in \mathcal{I}$ and let $\mathcal{E}_1$ be the set of edges adjacent to node 1. If $\mathcal{E}_1\subseteq I^*$, then
\[e_1^TA_{I^*} = (1-\lambda)e_1^T + \frac{\lambda}{n} e^T.
\]
\end{lemma}
\begin{proof}
    Let $I^* \in \mathcal{I}$, with $\mathcal{E}_1\subseteq I^*$ and let $d$ be a demand vector satisfying \eqref{genericD}. Let $f_{\mathcal{E}_1}^{(pl)}:= \sum_{i\in \mathcal{E}_1}{f_i^{(pl)}}$ and $f_{\mathcal{E}_1}^{(op)}:= \sum_{i\in \mathcal{E}_1}{f_i}$ be the total flow at edges adjacent to node 1 during the \textit{planning} and \textit{operational} phase, respectively. The assumption that $\mathcal{E}_1\subseteq I^*$ implies that the flow on edges adjacent to node 1 is equal to the \textit{operational} limit, which is the fraction $\lambda$ of flows in the \textit{planning} phase. This means that $f_{\mathcal{E}_1}^{(op)} = \lambda f_{\mathcal{E}_1}^{(pl)}$.  During the \textit{planning} phase, the flow into node 1 is equal to the shortage of demand at node 1, namely
    $$ f_{\mathcal{E}_1}^{(pl)} = d_1 - g_1^{(pl)} = 1 - \frac{1+\varepsilon}{n}.$$
    Therefore, we conclude that $f_{\mathcal{E}_1}^{(op)} = \lambda ( 1 - \frac{1+\varepsilon}{n}).$

Matrix $A_{I^*}$ maps $d$ to some generation vector $g$ for which the flow into node 1 is fixed. Since the flow into a node is given by a difference between the demand and generation, we know that $d_1 - g_1 = \lambda (1 - \frac{1+\varepsilon}{n})$. Hence, 
\begin{equation} \label{fixedFlowAI}g_1 = e_1^TA_{I^*}d = 1 - \lambda \left(1-\frac{1+\varepsilon}{n}\right).\end{equation}
Define $a_1:=e_1^TA_{I^*}$. Then,
\[e_1^TA_{I^*}d =  a_1(e_1 + \varepsilon\gamma) = a_1e_1 + \varepsilon a_1\gamma.\]
Then, \eqref{fixedFlowAI} yields
\[a_1e_1 + \varepsilon a_1\gamma = 1-\lambda + \frac{\lambda}{n} + \frac{\lambda}{n}\varepsilon.\]
Since the above statement must be true for all $\varepsilon$ and all $\gamma$, we find that $a_1e_1 = 1-\lambda + \frac{\lambda}{n}$. Moreover, $a_1\gamma = \frac{\lambda}{n}$ for all $\gamma$ if and only if $a_1$ is constant in every component except for the first one, which follows from the fact that $\gamma_1 = 0$. This implies that $a_1 = e_1^TA_{I^*} =  (1-\lambda)e_1 + \frac{\lambda}{n}e$.
\end{proof}
 We are now ready to prove Lemma \ref{firstColumnAgeneral}. In the proof, we show that the first row and the first column of the matrix $A_{I^*(d)}$ are equal, using Lemmas \ref{fullLine} and \ref{tightConstraints}. This, together with Lemma \ref{firstColumnA}, yields the result. 
\begin{proof}
Let $d$ be a demand vector as given in $\eqref{genericD}$ for some $\varepsilon\geq 0$ sufficiently small and let $I^*(d)$ be the corresponding hyperplane index set. Due to Lemma \ref{lemmaProjection}, we have that $g^*(d) = A_{I^*(d)}d$, with
\begin{equation}A_{I^*(d)} = \frac{1}{n}J + \begin{bmatrix} V_{I_1}^T & V_{I_2}^T\end{bmatrix}\left(
\begin{bmatrix}V_{I_1}\\ V_{I_2}\end{bmatrix}
\begin{bmatrix}V^T_{I_1} & V^T_{I_2}\end{bmatrix}\right)^{+}\begin{bmatrix}
(1-\lambda)V_{I_1} \\
(1+\lambda)V_{I_2}\end{bmatrix}\label{eqAI}.\end{equation}
Without loss of generality, assume that $\begin{bmatrix}V_{I_1}\\ V_{I_2}\end{bmatrix}$ has full rank. Otherwise, remove all dependent rows from $\begin{bmatrix}V_{I_1}\\ V_{I_2}\end{bmatrix}$. This action does not change the underlying subspace or the projection matrix; however, it simplifies the computations. This is because removing all dependent rows of $\begin{bmatrix}V_{I_1}\\ V_{I_2}\end{bmatrix}$ ensures that $\begin{bmatrix}V_{I_1}\\ V_{I_2}\end{bmatrix}
\begin{bmatrix}V^T_{I_1} & V^T_{I_2}\end{bmatrix}$ is a full-rank matrix. Consequently, its pseudoinverse is equal to its inverse.

Using the above representation of $A_{I^*(d)}$, we show that the first row and the first column of $A_{I^*(d)}$ are equal. Denote $\left(
\begin{bmatrix}V_{I_1}\\ V_{I_2}\end{bmatrix}
\begin{bmatrix}V^T_{I_1} & V^T_{I_2}\end{bmatrix}\right)^{-1} =: Q = \begin{bmatrix}
    Q_1 & Q_2\\
    Q_3 & Q_4
\end{bmatrix},$ with $Q_1\in \mathbb{R}^{|I_1|\times|I_1|}$,  $Q_2\in \mathbb{R}^{|I_1|\times|I_2|}$,
$Q_3\in \mathbb{R}^{|I_2|\times|I_1|}$, and $Q_4\in \mathbb{R}^{|I_2|\times|I_2|}$. Note that $Q$ is symmetric by construction, hence $Q_2 = Q_3^T$, $Q_1 = Q_1^T$, and $Q_4 = Q_4^T$. 

The first column of $A_{I^*(d)}$ is given by
\begin{align*}
    &A_{I^*(d)}e_1 =  \frac{1}{n}Je_1 + \begin{bmatrix} V_{I_1}^T & V_{I_2}^T\end{bmatrix}\begin{bmatrix}
        Q_1 & Q_2\\ Q_3 & Q_4
    \end{bmatrix}\begin{bmatrix}
(1-\lambda)V_{I_1} \\
(1+\lambda)V_{I_2}\end{bmatrix}e_1\\
&\hspace{0.6cm}= \frac{1}{n}e + \begin{bmatrix} V_{I_1}^T & V_{I_2}^T\end{bmatrix}\begin{bmatrix}
        Q_1 & Q_2\\ Q_3 & Q_4
    \end{bmatrix}\begin{bmatrix}
(1-\lambda)V_{I_1}e_1 \\
0\end{bmatrix}\\
&\hspace{0.6cm}= \frac{1}{n}e + (1-\lambda)V_{I_1}^TQ_1V_{I_1}e_1 + (1-\lambda) V_{I_2}^TQ_3V_{I_1}e_1,
\end{align*}
where in the second equality we applied the result of Lemma \ref{tightConstraints}.
Similarly, we derive an expression for the first row of $A_{I^*(d)}$:
\begin{align*}
    &e_1^TA_{I^*(d)} 
= \frac{1}{n}e^T + \begin{bmatrix} e_1^TV_{I_1}^T & 0\end{bmatrix}\begin{bmatrix}
        Q_1 & Q_2\\ Q_3 & Q_4
    \end{bmatrix}\begin{bmatrix}
(1-\lambda)V_{I_1} \\
(1+\lambda)V_{I_2}\end{bmatrix}\\
&\hspace{0.6cm}= \frac{1}{n}e^T + \begin{bmatrix} e_1^TV_{I_1}^TQ_1 & e_1^TV_{I_1}^TQ_2\end{bmatrix}\begin{bmatrix}
(1-\lambda)V_{I_1} \\
(1+\lambda)V_{I_2}\end{bmatrix}\\
&\hspace{0.6cm}= \frac{1}{n}e^T + (1-\lambda)e_1^TV_{I_1}^TQ_1V_{I_1} + (1+\lambda) e_1^TV_{I_1}^TQ_2V_{I_2}\\
&\hspace{0.6cm}= \frac{1}{n}e^T + (1-\lambda)e_1^TV_{I_1}^TQ_1V_{I_1} + (1+\lambda) e_1^TV_{I_1}^TQ_3^TV_{I_2}.
\end{align*}
Hence, we observe that $(e_1^TA_{I^*(d)})^T = A_{I^*(d)}e_1$ if and only if $V_{I_2}^TQ_3V_{I_1}e_1 = 0$. 

To show that $V_{I_2}^TQ_3V_{I_1}e_1 = 0$, we first write an explicit expression for $Q_3$. In general, the inverse of an invertible block matrix $\begin{bmatrix}
        P_1 & P_2\\ P_3 & P_4
    \end{bmatrix}$ is given by
    \[\begin{bmatrix}
        P_1 & P_2\\ P_3 & P_4
    \end{bmatrix}^{-1} = \begin{bmatrix}
        (P_1 - P_2P_4^{-1}P_3)^{-1} & -(P_1 - P_2P_4^{-1}P_3)^{-1}P_2P_4^{-1}\\ -(P_4 - P_3P_1^{-1}P_2)^{-1}P_3P_1^{-1} & (P_4 - P_3P_1^{-1}P_2)^{-1}
    \end{bmatrix}.\]
    Hence, $Q_3 = -Q_4V_{I_2}V_{I_1}^T(V_{I_1}V_{I_1}^T)^{-1}$ and $-V_{I_2}^TQ_4V_{I_1}e_1 = V_{I_2}^TQ_1V_{I_2}V_{I_1}^T(V_{I_1}V_{I_1}^T)^{-1}V_{I_1}e_1$.\\

Next, we show that $V_{I_1}^T(V_{I_1}V_{I_1}^T)^{-1}V_{I_1}e_1 = e_1 - \frac{1}{n}e$. For this purpose, we apply Lemma \ref{firstColumnA} to $\tilde{I} = I^*(d)\setminus\{m+1,\dots,2m\}.$ Note that the conditions of this lemma are met, namely that $\mathcal{E}_1\in \tilde{I}$, due to Lemma \ref{fullLine}. Hence, we obtain that $e_1^TA_{\tilde{I}} = (1-\lambda)e_1^T + \frac{\lambda}{n}e^T,$ with $$A_{\tilde{I}} = \frac{1}{n}J + (1-\lambda)V_{I_1}^T(V_{I_1}V_{I_1}^T)^{-1}V_{I_1}.$$
We observe that $A_{\tilde{I}}$ is a symmetric matrix. In particular:
\begin{align*}A_{\tilde{I}}^T &= \frac{1}{n}J^T + (1-\lambda)\left(V_{I_1}^T\left(V_{I_1}V_{I_1}^T\right)^{-1}V_{I_1}\right)^T \\
&= \frac{1}{n}J + (1-\lambda)V_{I_1}^T\left(\left(V_{I_1}V_{I_1}^T\right)^T\right)^{-1}V_{I_1}\\
&= \frac{1}{n}J + (1-\lambda)V_{I_1}^T\left(V_{I_1}V_{I_1}^T\right)^{-1}V_{I_1}= A_{\tilde{I}}.\end{align*}
The symmetry implies that $A_{\tilde{I}}e_1 = (1-\lambda)e_1 + \frac{\lambda}{n}e$. Hence,
\[V_{I_1}^T(V_{I_1}V_{I_1}^T)^{-1}V_{I_1}e_1 = e_1 + \frac{1}{1 - \lambda}\left(\frac{\lambda}{n} - \frac{1}{n}\right)e = e_1 - \frac{1}{n}e.\]
Lemma \ref{tightConstraints} shows that $g^*(d)$ lies on hyperplane $H_j$ for some $j\in \{m+1,\dots,2m\}$ if and only if $(Ve_1)_{j-m} = 0$. This implies that the first column of $V_{I_2}$ is equal to 0. In addition, recall that $V_{I_2}e = 0$. Hence,
    \[V_{I_2}V_{I_1}^T(V_{I_1}V_{I_1}^T)^{-1}V_{I_1}e_1 = V_{I_2}\left(e_1 - \frac{1}{n}e\right) = V_{I_2}e_1 = 0.\]
    Therefore, we conclude that 
    \[V_{I_2}^T Q_3 V_{I_1}e_1 = -V_{I_2}^TQ_4\left(V_{I_2}V_{I_1}^T(V_{I_1}V_{I_1}^T)^{-1}V_{I_1}e_1\right) =0. \]
    Hence, $A_{I^*(d)}e_1 = e_1^TA_{I^*(d)}$. Again, due to Lemma \ref{fullLine}, we know that $\mathcal{E}_1\subseteq I^*(d)$. Therefore, we can apply Lemma \ref{firstColumnA} to $I^*(d)$, which  yields that
    \[A_{I^*(d)}e_1 = (1-\lambda)e_1 + \frac{\lambda}{n}e.\]
\end{proof}

\subsection{Proof of Proposition \ref{propositionKKT}}
In this section, we prove Proposition \ref{propositionKKT}, which provides a characterization of the optimal generation $g^*(d)$. The proof uses the Karush-Kuhn-Tucker conditions to prove the optimality of the proposed solution of the DC-OPF problem \eqref{DC-OPF}.

\begin{proof}
Let $\gamma = (0,\gamma_2,\dots,\gamma_n)$, with $e^T\gamma = 1$, $\gamma> 0$. Fix $\hat{\varepsilon}>0$ small enough. Using Lemma \ref{lemmaProjection}, for $\hat{d}:= e_1 + \hat{\varepsilon}\gamma$, we have that the optimal solution of \eqref{DC-OPF} is given by $g^\ast(\hat{d}) = \hat{A}_{I^\ast(\hat{d})}\hat{d}$, with $\hat{A}_{I^\ast(d)}:=\hat{A}_{I^\ast}$. Note that $\hat{\varepsilon}$ is fixed, therefore $\hat{A}_{I^\ast}$ depends only on $\gamma$. 

The main idea behind the proof is to use the Karush–Kuhn–Tucker (KKT) conditions associated with $g^\ast(\hat{d})$ to establish that $\hat{A}_{I^\ast}d$ is the optimal solution of the optimization problem \eqref{DC-OPF} for any $d(\varepsilon,\gamma)$ with $\varepsilon<\hat{\varepsilon}$. The optimality of $\hat{A}_{I^\ast}d$ is ensured by finding suitable Lagrange multipliers that satisfy the corresponding KKT conditions. 

We begin by proving that strong duality holds for our optimization problem using Slater's condition \cite{boyd2004convex}. Consider a solution $g = d$, where each node generates its own demand. This solution satisfies the equality constraint \eqref{balance} because $e^Tg = e^Td$. It remains to show that, for this choice of $g$, all inequality constraints \eqref{edgeConstraints} are not tight. Let $l\in \mathcal{E}$ and suppose that $(Vd)_l>0$. Then, \[(1-\lambda)(Vd)_l< (Vd)_l = (Vg)_l<(1+\lambda)(Vd)_l.\]
Thus, the constraints corresponding to edge $l$ are not tight. Furthermore, if $(Vd)_l = 0$, then $(1-\lambda)(Vd)_l = (1+\lambda)(Vd)_l = 0$, so the inequality constraints together form an equality constraint. Hence, by Slater's condition, we conclude that strong duality holds.\\

To prove that $g^\ast(d)$ is the optimal solution to the optimization problem with input $d$, we need to find suitable Lagrange multipliers $\mu$, $\nu$ and $\delta$ that satisfy the following KKT conditions:
\begin{equation}\label{eq1}g^\ast(d) + V^T(\nu - \mu) + \delta e = 0,\end{equation}
\begin{equation}\label{eq2}\mu\geq 0, \nu\geq 0,\end{equation}
\begin{equation}\label{eq3}\mu_i\Big((1-\lambda)Vd - Vg^\ast(d)\Big)_i = 0,~~i = 1,\dots,m,\end{equation}
\begin{equation}\label{eq4}\nu_i\Big((1+\lambda)Vd - Vg^\ast(d)\Big)_i = 0,~~i = 1,\dots,m,\end{equation}
\begin{equation}\label{eq5}(1-\lambda)Vd\leq  Vg^\ast(d) \leq (1+\lambda)Vd,\end{equation}
\begin{equation}\label{eq6}e^Tg^\ast(d) = e^Td.\end{equation} 
Generation $g^\ast(\hat{d})$ is the optimal solution of \eqref{DC-OPF} for input $\hat{d}$. Therefore, there exists a set of Lagrange multipliers $\hat{\mu}\geq 0$, $\hat{\nu}\geq 0$ and $\hat{\delta}$ that satisfy Conditions \eqref{eq1} -- \eqref{eq6} for $d = \hat{d}$. In what follows, we use these Lagrange multipliers, to show that $\hat{A}_{I^\ast}d$ is the optimal solution of \eqref{DC-OPF} for every $d = e_1 + \varepsilon \gamma$ with $\varepsilon = \alpha\hat{\varepsilon}$ and $\alpha\in (0,1)$ by demonstrating that $\hat{A}_{I^\ast}d$ satisfies all KKT conditions \eqref{eq1} -- \eqref{eq6}.\\


\noindent\underline{Condition (\ref{eq6})}:
Equation (\ref{eq6}) applied to $\hat{d}$ gives us
\[e^T\hat{a}_1 + e^T\hat{A}_{I^\ast}(\hat{\varepsilon} \gamma)= e^T\hat{A}_{I^\ast}\hat{d} = e^Tg^\ast(\hat{d}) \overset{\eqref{eq6}}{=} e^T\hat{d} =  1+\hat{\varepsilon},\]
where $\hat{a}_1:= \hat{A}_{I^\ast}e_1$. From Lemma \ref{firstColumnAgeneral}, we know that $\hat{a}_1 = (1-\lambda)e_1 + \frac{\lambda}{n}e$. We obtain that $e^T\hat{a}_1 = 1$; hence $e^T\hat{A}_{I^\ast} \gamma = 1$. Consequently,
\[e^Tg^\ast(d) = e^T\hat{A}_{I^\ast}d:= e^T\hat{a}_1 + e^T\hat{A}_{I^\ast}(\varepsilon \gamma)= 1+\varepsilon = e^Td.\]
This shows that Condition (\ref{eq6}) is satisfied for demand vector $d$. \\

\noindent\underline{Condition (\ref{eq5})}:
Next, we verify that Condition (\ref{eq5}) holds. Equation (\ref{eq5}) applied to $\hat{d}$ gives us
\[(1-\lambda)V\hat{d} \leq  Vg^\ast(\hat{d}) \leq (1+\lambda)V\hat{d}.\]
For the expression in the middle, we find
\[Vg^\ast(\hat{d}) = V\hat{A}_{I^\ast}(e_1 + \hat{\varepsilon}\gamma) = V(\hat{a}_1 + \hat{\varepsilon}\hat{A}_{I^\ast}\gamma) = V\big((1-\lambda)e_1 + \frac{\lambda}{n}e + \hat{\varepsilon}\hat{A}_{I^\ast}\gamma\big) = V\big((1-\lambda)e_1 + \hat{\varepsilon}\hat{A}_{I^\ast}\gamma\big),\]
where in the last line follows from the property $Ve = 0$ of the PTDF matrix. Therefore, Equation \eqref{eq5} becomes
\begin{equation*}(1-\lambda)V(e_1 + \hat{\varepsilon}\gamma) \leq  V((1-\lambda)e_1 + \hat{\varepsilon}\hat{A}_{I^\ast}\gamma) \leq (1+\lambda)V(e_1 + \hat{\varepsilon}\gamma).\end{equation*}
We first focus on the left-hand side inequality, which, after simplification, yields \[(1-\lambda)V(\hat{\varepsilon}\gamma) \leq  V(\hat{\varepsilon}\hat{A}_{I^\ast}\gamma).\]
Multiplying both sides by $\varepsilon/\hat{\varepsilon}$ and adding $(1-\lambda)Ve_1$ yields the left inequality of (\ref{eq5}) for demand $d$. Rewriting the right-hand side inequality gives us
\[-\lambda Ve_1 +\hat{\varepsilon}V\hat{A}_{I^\ast}\gamma\leq \lambda Ve_1 + (1+\lambda)\hat{\varepsilon}V\gamma,\] and hence, \[\hat{\varepsilon}\big(V\hat{A}_{I^\ast}\gamma -(1+\lambda)V\gamma\big)\leq 2\lambda Ve_1.\]

If $(V\hat{A}_{I^\ast}\gamma -(1+\lambda)V\gamma)_i\leq 0$ for some $i = 1,
\dots,m$, then inequality $\varepsilon\big(V\hat{A}_{I^\ast}\gamma -(1+\lambda)V\gamma\big)_i\leq0\leq 2\lambda (Ve_1)_i $ is satisfied irrespective of the value of $\varepsilon$, since $Ve_1\geq 0$. On the other hand, if $(V\hat{A}_{I^\ast}\gamma -(1+\lambda)V\gamma)_i\geq 0$, then
\[\varepsilon\big(V\hat{A}_{I^\ast}\gamma -(1+\lambda)V\gamma\big)_i\leq\hat{\varepsilon}\big(V\hat{A}_{I^\ast}\gamma -(1+\lambda)V\gamma\big)_i\leq 2\lambda (Ve_1)_i.\]
Hence, Condition (\ref{eq5}) is satisfied for demand  $d$.\\

For Conditions \eqref{eq1} -- \eqref{eq4}, we choose Lagrange multipliers $\nu = (1-\alpha)\tilde{\nu} + \alpha\hat{\nu}$, $\mu = (1-\alpha)\tilde{\mu} + \alpha\hat{\mu}$ and $\delta = \alpha\hat{\delta}+ (1-\alpha)\tilde{\delta},$ where $\tilde{\nu}$, $\tilde{\mu}$ and $\tilde{\delta}$ are the Lagrange multiplies corresponding to the optimization problem with $d = e_1$, for which the optimal solution is known. In particular, $\tilde{\delta} := -\frac{1}{n}$, $\tilde{\nu} := 0$ and $\tilde{\mu} := (1-\lambda)Ce_1$ \cite[Lemma IV.2]{Nesti_2020}.
\\

\noindent\underline{Condition (\ref{eq2})}:
Since $\tilde{\nu}$, $\tilde{\mu}$, $\hat{\nu}$, $\hat{\mu}$ satisfy \eqref{eq2}, by definition, $\mu$ and $\nu$ also satisfy Condition \eqref{eq2}. \\

\noindent\underline{Condition (\ref{eq1})}:
Substituting $g^*(d) = \hat{a}_1 + \hat{A}_{I^*}(\varepsilon\gamma)$ into Condition \eqref{eq1} yields
\begin{equation*}\hat{a}_1 + \varepsilon \hat{A}_{I^\ast}\gamma + V^T(\nu - \mu) + \delta e = 0, \end{equation*}
and hence
\begin{equation}(1-\lambda)e_1 + \frac{\lambda}{n}e + \varepsilon  \hat{A}_{I^\ast}\gamma + V^T(\nu - \mu) + \delta e = 0.\label{alpha}\end{equation}

We know that \eqref{alpha} holds for $\hat{\mu}$, $\hat{\nu}, \hat{\delta}$ and $\hat{\varepsilon}$. Moreover, \eqref{alpha} holds for $\tilde{\mu}$, $\tilde{\nu}$, $\tilde{\delta}$ and $\tilde{\varepsilon} = 0$. A weighted sum of these two equalities with weights $\alpha$ and $(1-\alpha)$, respectively, yields the following equality:
\begin{multline}0 = (1-\lambda)e_1 + \frac{\lambda}{n}e + \varepsilon  \hat{A}_{I^\ast}\gamma + V^T\Big(\alpha\hat{\nu} + (1-\alpha)\tilde{\nu} - \alpha\hat{\mu} - (1-\alpha)\tilde{\mu}\Big) + \Big(\alpha\hat{\delta} + (1-\alpha)\tilde{\delta}\Big)e\\
= (1-\lambda)e_1 +\frac{\lambda}{n}e + \varepsilon  \hat{A}_{I^\ast}\gamma + V^T\big(\nu - \mu\big) + \delta e.\end{multline}
Hence, Condition \eqref{eq1} is satisfied for $d$.

\noindent\underline{Condition (\ref{eq4}}):
To be able to say more about $
\hat{\nu}$, we investigate Equation (\ref{eq4}) for $\hat{d}$. It can be rewritten as follows:
\[\hat{\nu}_i\Big((1+\lambda)Ve_1 + (1+\lambda)\hat{\varepsilon}V\gamma - (1-\lambda)Ve_1 -\hat{\varepsilon}V \hat{A}_{I^\ast}\gamma\Big)_i = 0, ~~i = 1,\dots,m.\]
Thus, we have
\[\hat{\nu}_i\Big(2\lambda Ve_1 + (1+\lambda)\hat{\varepsilon}V\gamma  -\hat{\varepsilon}V \hat{A}_{I^\ast}\gamma\Big)_i = 0,~~i = 1,\dots,m.\]
Suppose that $(Ve_1)_i > 0 $. Then, since $\hat{\varepsilon}$ is small enough, Lemma \ref{tightConstraints} yields $\Big(2\lambda Ve_1 + (1+\lambda)\hat{\varepsilon}V\gamma  -\hat{\varepsilon}V \hat{A}_{I^\ast}\gamma\Big)_i>0$, which implies that $\hat{\nu}_i = 0$. In addition, we know that $\tilde{\nu}_i = 0$, which yields that $\nu_i = 0$.

Now, suppose that $(Ve_1)_i = 0$. If $\hat{\nu}_i = 0,$ then again, $\nu_i = 0$. On the other hand, if $\hat{\nu}_i\neq 0$, then 
\begin{equation*}\label{contradict}2\lambda (Ve_1)_i + (1+\lambda)\hat{\varepsilon}(V\gamma)_i  -\hat{\varepsilon}(V \hat{A}_{I^\ast}\gamma)_i = (1+\lambda)\hat{\varepsilon}(V\gamma)_i  -\hat{\varepsilon}(V \hat{A}_{I^\ast}\gamma)_i = 0.\end{equation*}
Multiplying the above equation with $\alpha$ yields
\[2\lambda (Ve_1)_i + (1+\lambda)\varepsilon(V\gamma)_i  -\varepsilon(V \hat{A}_{I^\ast}\gamma)_i = 0,\]
which means that the $i$-th equation of (\ref{eq4}) is satisfied. Altogether, this implies that Equation (\ref{eq4}) holds for all $i$.\\

\noindent\underline{Condition (\ref{eq3}):} Clearly, if $\mu_i = 0$, then the constraint is satisfied for $i$. Therefore, suppose that $\mu_i>0$. This means that either $\tilde{\mu}_i>0$ or $\hat{\mu}_i>0$ (or both). We know that $\tilde{\mu} = Ce_1$ with $C$ the incidence matrix of graph $\mathcal{G}$. Therefore, $\tilde{\mu}_i>0$ if and only if $i\in \mathcal{E}_1$. Then by Lemma \ref{fullLine}, Equation (\ref{eq3}) holds for such a choice of $i$. \\

It remains to consider cases when $\hat{\mu}_i>0$, $i\neq \mathcal{E}_1$. In this case, Equation (\ref{eq3}) for $\hat{\varepsilon}$ gives us
\[(1-\lambda)Ve_1 + (1-\lambda)\hat{\varepsilon}V\gamma = (1-\lambda)Ve_1 + \hat{\varepsilon}V \hat{A}_{I^\ast}\gamma.\]
Hence,
\[ (1-\lambda)\hat{\varepsilon}V\gamma =  \hat{\varepsilon}V \hat{A}_{I^\ast}\gamma.\]
Multiplying the above by $\alpha$ and adding factor $(1-\lambda)Ve_1$ gives us
\[(1-\lambda)Ve_1 + (1-\lambda)\varepsilon V\gamma = (1-\lambda)Ve_1 + \varepsilon V \hat{A}_{I^\ast}\gamma,\]
so (\ref{eq3}) holds for this choice of $i$ as well. \\

Since all KKT conditions \eqref{eq1} -- \eqref{eq6} are satisfied, we conclude that $g^*(d) =  \hat{A}_{I^\ast}d$ for all $\varepsilon >0$ sufficiently small. Lastly, we set $\phi(\gamma) := I^*(\hat{d})$, which completes the proof. \end{proof}

\subsection{Proof of Lemma \ref{cascadeEpsilon}} 
The proof of Lemma \ref{cascadeEpsilon} relies on the continuity of the demand and generation with respect to $\varepsilon>0$ at every step of the cascade. After showing the continuity, we use Taylor's expansion theory to show that the edges maximizing the relative flow exceedance at every step of the cascade are the same for all $\varepsilon>0$ sufficiently small.
\begin{proof}
Assume that the node with the largest demand has label $i$, that edge $l$ fails first, and fix $\gamma \in \Gamma$. To prove the lemma, we recursively show that the cascade sequence is independent of $\varepsilon$ in the neighborhood of 0. 

Consider the graph after the first step of the cascade.  Due to Proposition \ref{propositionKKT}, we know that the initial demand vector $d$ and the optimal generation vector $g^*(d)$ are affine functions of $\varepsilon$ for all $\varepsilon>0$ small, and hence, they are continuously differentiable. Next, we argue that $d^{(2)}$ and $g^{(2)}$ are continuous with respect to $\varepsilon$. Consider a connected component $\mathcal{C} = (\mathcal{N}_\mathcal{C}, \mathcal{E}_\mathcal{C})$ of $\mathcal{G}^{(2)}$. Equation \eqref{balanceRestoration} implies that for $\theta^{(2)} = \sum_{i\in \mathcal{N}_\mathcal{C}} d_i/\sum_{i\in \mathcal{N}_\mathcal{C}} g_i^*(d)$: 
\begin{enumerate}
    \item[(i)] if there is excess of demand in $\mathcal{C}$, i.e. $\theta^{(2)}>1$, then $d_i^{(2)} = d_i/\theta^{(2)}$ and $g_i^{(2)} = g_i^*(d)$ for $i\in \mathcal{N}_\mathcal{C}$,
    \item[(ii)] if there is excess of generation in $\mathcal{C}$, i.e. $\theta^{(2)}<1$, then $d_i^{(2)} = d_i$ and $g_i^{(2)} = \theta^{(2)}g_i^*(d)$ for $i\in \mathcal{N}_\mathcal{C}$,
    \item[(iii)] if there is a balance of generation and demand in $\mathcal{C}$, i.e. $\theta^{(2)}=1$, then  $d_i^{(2)} = d_i$ and $g_i^{(2)} = g_i^*(d)$ for $i\in \mathcal{N}_\mathcal{C}$.
\end{enumerate} Due to the linearity of the initial demand and generation with respect to $\varepsilon$, it follows that $\theta^{(2)}$ is continuous in $\varepsilon$. Hence, we conclude that $d^{(2)}$ and $g^{(2)}$ are also continuous in $\varepsilon$. 

Next, we argue that the maximizers of the relative exceedance $\psi^{(2)}$ are independent of $\varepsilon$. First, we observe that for any $e\in \mathcal{E}^{(1)}$, the relative exceedance $\psi_e^{(2)}$ is a rational function $\frac{P(\varepsilon)}{Q(\varepsilon)}$ where both $P$ and $Q$ are some polynomials in $\varepsilon$ of degree $k = 2$. This follows from the fact that both flows and emergency flow limits are linear functions of demand and generation and from the above definition of the demand and generation. Given this, we consider two cases:
\begin{enumerate}
    \item[Case 1:]   There is only one edge $e^*$ that maximizes the exceedance at $\varepsilon = 0$. Then, due to continuity of rational functions on their domain, $e^*$ is the exceedance maximizer for all $\varepsilon>0$ sufficiently small.
    \item[Case 2:] There is a set of edges $\mathcal{E}^*\subset \mathcal{E}$, with $|\mathcal{E}^*|\geq 2$ such that each $e^*\in \mathcal{E}^*$ maximizes the exceedance at $\varepsilon = 0$. In this case, we need to compare the speed of growth of \textit{unique} functions $\psi_{e^*}^{(2)}$ at $\varepsilon = 0$. This is possible by comparing the derivatives, since rational functions are infinitely often differentiable. If the first derivative test is inconclusive, namely, there are multiple functions $\psi_{e^*}^{(2)}$ that maximize $\psi_{e^*}'^{(2)}$ at $\varepsilon = 0$, we continue the test with higher order derivatives until the test is conclusive. We emphasise that there could be multiple edges $e^*$ that have the exact same exceedance for all $\varepsilon\geq 0$ sufficiently small, therefore, we only compare the derivatives of \textit{unique} functions $\psi_{e^*}^{(2)}$. As the result of the derivative test, we obtain some function $\psi^*$, which has the fastest growth amongst functions $\psi_{e^*}^{(2)}$. Hence, the set of exceedance maximizers is unique for all $\varepsilon>0$ sufficiently small and given by $\{e^*\in \mathcal{E}^*~:~ \psi_{e^*}^{(2)} = \psi^*\}$.
\end{enumerate}
Since in both cases, the set of exceedance maximizers is independent of $\varepsilon$, we conclude that the graph after the second cascade step, $\mathcal{G}^{(3)}$, is also independent of $\varepsilon$ in this setting. 

Finally, consider the graph after $r$ cascade steps for some $r\in \mathbb{N}$. We can use the same line of arguments to argue that if $\mathcal{G}^{(r)}$ is independent of $\varepsilon$, then the resulting graph $\mathcal{G}^{(r+1)}$ is also independent of $\varepsilon$. Hence, by induction, it follows that in a sufficiently small neighborhood of $\varepsilon = 0$, the cascade sequence and the resulting graph $\mathcal{G}^{(\text{end})}$ does not depend on $\varepsilon$. 

\end{proof}
\subsection{Proof of Lemma \ref{asymptoticS}}

In order to prove the asymptotic convergence of $\mathbb{P}(Y_1>x, \sum_{j = 2}^n Y_j\leq \varepsilon Y_1,Y\in \mathcal{Y}| N = i, L^{(1)} = l)$, we first rewrite the expression in a suitable manner and then construct upper and lower bounds. Consequently, we show that both bounds are equal in the limit of $x\rightarrow \infty$.
\begin{proof}
First, we observe that the distribution of $Y$ is independent of the first edge failure $L^{(1)} = l$. Moreover, because $X_i$'s are independent and identically distributed, conditioning on the label of the node with the largest demand $N = i$ does not change the distribution of $Y$. Therefore, the conditioning only influences the set $Y(i,l,z,I^*)$. For simplicity of notation, we write $\mathcal{Y} = Y(i,l,z,I^\ast)$, as the dependence on $i,l,z$, and $I^\ast$ is not exploited. In the proof, it is only relevant that the initial set of demands $Y$ is in \textit{some} well-defined subset $\mathcal{Y}$, but the shape of this subset, specified through variables $i$, $l$, $z$, and $I^*$ does not play a role.  it becomes relevant in the proof of Theorem \ref{mainThm}, where we study the result of cascade as a function of the initial demand. 

We observe that we can write the event $\{Y_1>x\}$ as a union of disjoint events $\{X_1>x, X_2< X_1 , \dots, X_n< X_1\}\cup\dots \cup \{X_n>x, X_1< X_n,\dots, X_{n-1}< X_n\}$, where we implicitly use the fact that the probability of two continuous random variables being equal is 0. Using this, we write
\begin{align*}
    &\mathbb{P}(Y_1>x, \sum_{j = 2}^n Y_j\leq \varepsilon Y_1,Y\in \mathcal{Y}| N = i, L^{(1)} = l) \\&\hspace{0.6cm}= \sum_{k = 1}^n \mathbb{P}(X_k>x, X_1\leq X_k, \dots, X_n\leq X_k, \sum_{j = 2}^n Y_j\leq \varepsilon Y_1,Y\in \mathcal{Y}| N = i, L^{(1)} = l) & (X_k \text{ is the maximum})\\
    &\hspace{0.6cm}= \sum_{k = 1}^n \mathbb{P}(X_1>x, X_2\leq X_1, \dots, X_n\leq X_1, \sum_{j = 2}^n X_j\leq \varepsilon X_1,Y\in \mathcal{Y} | N = i, L^{(1)} = l)& (X_k \text{'s are i.i.d.})\\
    &\hspace{0.6cm}=n\mathbb{P}(X_1>x,  X_2\leq X_1, \dots, X_n\leq X_1, \sum_{j = 2}^n X_j\leq \varepsilon X_1, Y\in \mathcal{Y}| N = i, L^{(1)} = l).
\end{align*}
For $\varepsilon<1$ we have that $\{X_1\leq X_k, \dots, X_n\leq X_k\}\supset \{\sum_{j = 1, j\neq k}^n X_j<\varepsilon X_k\}$ because $X_j$'s take positive values. This means that we can omit the requirements $X_2\leq X_1, \dots X_n\leq X_1$ in the expression above without changing the probability of the event in question. In particular, we obtain
\begin{align*}
    \mathbb{P}(Y_1>x, &\sum_{j = 2}^n Y_j\leq \varepsilon Y_1,Y\in \mathcal{Y} | N = i, L^{(1)} = l) = n\mathbb{P}(X_1>x, \sum_{j = 2}^n X_j\leq \varepsilon X_1, Y\in \mathcal{Y} | N = i, L^{(1)} = l).
\end{align*}
Our goal is to derive upper and lower bounds for this probability, converging to the same limit as $x\rightarrow\infty$, under appropriate scaling. To this end, we write the event $\{\sum_{j = 2}^n X_j\leq \varepsilon X_1\}$ as a union of two disjoint events, namely $\{\sum_{j = 2}^n X_j\leq \varepsilon x\}\cup\{\varepsilon x < \sum_{j = 2}^n X_j\leq \varepsilon X_1\}$. Then,
\begin{multline}\label{twoSummands}
    \mathbb{P}\Big(Y_1>x, \sum_{j = 2}^n Y_j\leq \varepsilon Y_1,Y\in \mathcal{Y}| N = i, L^{(1)} = l \Big)
    =n\mathbb{P}\Big(X_1>x, \sum_{j = 2}^n X_j\leq \varepsilon x, Y\in \mathcal{Y}| N = i, L^{(1)} = l\Big)
    \\\hspace{0.6cm}+n\mathbb{P}\Big(X_1>x,  \varepsilon x<\sum_{j = 2}^n X_j\leq \varepsilon X_1, Y\in \mathcal{Y}\Big| N = i, L^{(1)} = l).
    \end{multline}
When event $\{X_1>x, \sum_{j = 2}^n X_j\leq \varepsilon x, Y\in \mathcal{Y}| N = i, L^{(1)} = l\}$ is considered,  $X_1$ is the maximum of $X_1,\dots, X_n$, and therefore $Y = X$. This means that $\{Y\in \mathcal{Y}| N = i, L^{(1)} = l\}$ does not depend on $X_1$ in this particular setting. Hence, using the independence of $X_i$'s, we obtain
\begin{equation}\mathbb{P}\Big(X_1>x, \sum_{j = 2}^n X_j\leq \varepsilon x, Y\in \mathcal{Y}| N = i, L^{(1)} = l\Big) = \mathbb{P}\left(X_1>x\right) \mathbb{P}\Big(\sum_{j = 2}^n X_j\leq \varepsilon x, Y\in \mathcal{Y}| N = i, L^{(1)} = l\Big).\label{firstSummand}\end{equation}
Upper and lower bounds are obtained by approximating the second summand of \eqref{twoSummands}, i.e., \begin{equation}\label{secondTerm}\mathbb{P}\Big(X_1>x,  \varepsilon x<\sum_{j = 2}^n X_j\leq \varepsilon X_1, Y\in \mathcal{Y}| N = i, L^{(1)} = l\Big).\end{equation} For a lower bound, we approximate \eqref{secondTerm} with 0, and for an upper bound with $\mathbb{P}\Big(X_1>x,  \varepsilon x<\sum_{j = 2}^n X_j \Big)$. Clearly, this is an upper bound as we remove the condition on $Y$ and $X_1$. Then, leveraging the independence of $X_i$'s we obtain
\begin{equation}
0\leq \mathbb{P}\Big(X_1>x,  \varepsilon x<\sum_{j = 2}^n X_j\leq \varepsilon X_1, Y\in \mathcal{Y}| N = i, L^{(1)} = l\Big)\leq  \mathbb{P}\Big(X_1>x\Big)\mathbb{P}\Big(\varepsilon x<\sum_{j = 2}^n X_j \Big).  \label{bounds}  
\end{equation}
Using Equations \eqref{twoSummands}, \eqref{firstSummand}, and \eqref{bounds}, altogether we obtain that
\begin{align*} &n\mathbb{P}\left(X_1>x\right) \mathbb{P}\Big(\sum_{j = 2}^n X_j\leq \varepsilon x, Y\in \mathcal{Y}| N = i, L^{(1)} = l\Big) \\
&\hspace{0.6cm}\leq \mathbb{P}\Big(Y_1>x, \sum_{j = 2}^n Y_j\leq \varepsilon Y_1,Y\in \mathcal{Y} | N = i, L^{(1)} = l \Big) \\&\hspace{0.6cm}\leq n \mathbb{P}\left(X_1>x\right) \left(\mathbb{P}\Big(\sum_{j = 2}^n X_j\leq \varepsilon x, Y\in \mathcal{Y}| N = i, L^{(1)} = l\Big) + \mathbb{P}\Big(\varepsilon x<\sum_{j = 2}^n X_j\Big)\right).\end{align*}

We observe that
\(\lim_{x\rightarrow \infty} \mathbb{P}\Big(\sum_{j = 2}^n X_j\leq \varepsilon x, Y\in \mathcal{Y}| N = i, L^{(1)} = l\Big) = \mathbb{P}(Y\in \mathcal{Y}| N = i, L^{(1)} = l).\) Note that this probability is positive, by the condition of the lemma. Moreover, as $X_i$'s are i.i.d. Pareto-tailed, $\mathbb{P}(\sum_{j = 2}^n X_j>\varepsilon x) = \mathcal{O}(x^{-\alpha}) = o(1)$. Hence, both the lower and upper bound of $\mathbb{P}\Big(Y_1>x, \sum_{j = 2}^n Y_j\leq \varepsilon Y_1,Y\in \mathcal{Y}| N = i, L^{(1)} = l \Big)/(n\mathbb{P}(X_1>x))$ converge to $\mathbb{P}(Y\in \mathcal{Y}| N = i, L^{(1)} = l)>0$ as $x\rightarrow \infty$. In other words, 
\[\mathbb{P}(Y_1>x, \sum_{j = 2}^n Y_j\leq \varepsilon Y_1,Y\in \mathcal{Y}| N = i, L^{(1)} = l)\sim n\mathbb{P}(X_1>x)\mathbb{P}(Y\in \mathcal{Y}| N = i, L^{(1)} = l).\]
\end{proof}

\section{Uniqueness of flow exceedance} \label{emergency}
In this section, we discuss the uniqueness of flow exceedance at an arbitrary step of the cascade. Our main goal is to show that there exist graphs for which, with positive probability, the maximizers of the flow exceedance are not always unique. To this end, we first identify a set of conditions that need to be satisfied for two edges to have the same flow exceedance. Then, we present an example of a graph that satisfies the aforementioned conditions and show that there exists a cascade leading to non-unique flow exceedance maximizers. This implies that the assumptions of \cite[Theorem IV.12]{Nesti_2020} are not merely technical, but essential, and thus our theorem offers a relevant generalization. In addition, the example demonstrates that a tie-breaking rule, as given in Definition \ref{tie-breakingDef}, is crucial to ensure that the cascade process is well-defined. 


Consider a graph $\mathcal{G} = (\mathcal{N},\mathcal{E})$ and demand $d(\varepsilon,\gamma)$ as defined in \eqref{genericD}. We assume that $d$ meets the conditions of Proposition \ref{propositionKKT}, which means that the optimal generation is given by
\[g^\ast(d) = A_{\phi(\gamma)}d.\]
We consider the setting after $r$ edge failures. For $l\in \mathcal{E}^{(r+1)}$ let $v_l:= e_l^TV$ and $v_l^{(r+1)}:= e_l^TV^{(r+1)}$ denote the $l$-th row of $V$ and $V^{(r+1)}$, respectively. For simplicity, we assume that $\mathcal{G}^{(r+1)}$ is still a connected graph. In such a case, there is no demand and generation reduction in the network so $g^{(r+1)}(d) = g^\ast(d)$ and $d^{(r+1)} = d$. In this setting, we derive the following lemma, which provides necessary and sufficient conditions for equal flow exceedance of two edges for all $\varepsilon\geq0$.
\begin{lemma}\label{conditions}
    Suppose that $\mathcal{G}^{(r+1)}$ is a connected graph. Edges $j$ and $k\in \mathcal{E}^{(r+1)}$ have equal flow exceedance for all $\varepsilon\geq 0$ sufficiently small if and only if the following equations are simultaneously satisfied
    \begin{numcases}{}
    e_1^TQ^{(r+1)}\left(A_{\phi(\gamma)} - I\right)e_1 = 0, \label{const}\\
    e_1^T\big[Q^{(r+1)}\left(A_{\phi(\gamma)} - I\right) + \left(A_{\phi(\gamma)} - I\right)^T\left(Q^{(r+1)}\right)^T\big]\gamma = 0 \label{lin},\\
    \gamma^T Q^{(r+1)}\left(A_{\phi(\gamma)} - I\right)\gamma = 0\label{quad},
\end{numcases}
where $$Q^{(r+1)}:=\begin{cases}
v^T_jv_k^{(r+1)} - v^T_kv_j^{(r+1)} & \text{if } f_k^{(r+1)}(d)\cdot f_j^{(r+1)}(d) \geq 0, \\
v^T_jv_k^{(r+1)} + v^T_kv_j^{(r+1)} & \text{otherwise,}
\end{cases}$$ and $f^{(r+1)}(d)$ is the flow vector corresponding to demand $d$ after $r$ edge failures, i.e. $f^{(r+1)}(d):= V^{(r+1)}(d^{(r+1)} - g^{(r+1)})$.
\end{lemma}
To prove the lemma, we equate the flow exceedances of two arbitrary edges and write it as a polynomial in $\varepsilon$. We conclude that this polynomial is equal to 0 for all $\varepsilon\geq 0$ if and only if all of its coefficients are equal to 0.
\begin{proof}
Suppose that $j,k\in \mathcal{E}^{(r+1)}$ have the same relative flow exceedance for all $\varepsilon\geq 0$ sufficiently small. The relative flow exceedance of an arbitrary edge $i\in \mathcal{E}^{(r+1)}$ after $r$ edge failures is given by 
\[\psi_{i}^{(r+1)}(d) = \frac{|f^{(r+1)}_{i}(d)|}{F_{i}(d)}.\]
Rewriting the expression for $\psi_i^{(r+1)}$, we find
\begin{align*}
    \psi_{i}^{(r+1)}(d) &= \frac{|f_i^{(r+1)}(d)|}{F_{i}(d)} = \frac{|\big(V^{(r+1)}(g^{(r+1)}(d) - d^{(r+1)})\big)_i|}{\lambda^\ast(Vd)_{i}}= \frac{|v_i^{(r+1)}\big(A_{\phi(\gamma)} - I\big)(e_1 +\varepsilon\gamma)|}{\lambda ^\ast v_{i}(e_1 + \varepsilon\gamma)}.
\end{align*}\\
Edges $j$ and $k$ have the same exceedance if and only if $\psi_j^{(r+1)}(d) = \psi_k^{(r+1)}(d)$. 

First, we consider the case when $f_k^{(r+1)}(d)\cdot f_j^{(r+1)}(d) \geq 0$. For $\psi_k^{(r+1)}(d) = \psi_j^{(r+1)}(d)$ to be equal, it must hold that
\begin{multline}\Big(v_k^{(r+1)}A_{\phi(\gamma)}e_1 + v_k^{(r+1)}A_{\phi(\gamma)}\gamma\varepsilon -v_k^{(r+1)}e_1 -v_k^{(r+1)} \gamma\varepsilon\Big)\Big( v_{j}(e_1 + \varepsilon\gamma)\Big) \\
=\Big(v_j^{(r+1)}A_{\phi(\gamma)}e_1 + v_j^{(r+1)}A_{\phi(\gamma)}\gamma\varepsilon -v_j^{(r+1)}e_1 -v_j^{(r+1)}\gamma\varepsilon\Big)\Big(v_{k}(e_1 + \varepsilon\gamma)\Big).\label{uniquenessEq}\end{multline}
We rewrite Equation (\ref{uniquenessEq}) as a polynomial in $\varepsilon$ and obtain
\begin{multline*}
0 = e_1^T\left(v_{j}^Tv_k^{(r+1)} - v_{k}^Tv_j^{(r+1)}\right)\left(A_{\phi(\gamma)} - I\right)e_1 + e_1^T\left(v_{j}^Tv_k^{(r+1)} - v_{k}^Tv_j^{(r+1)}\right)\left(A_{\phi(\gamma)} - I\right)\gamma \varepsilon\\
+ \gamma^T\left(v_{j}^Tv_k^{(r+1)} - v_{k}^Tv_j^{(r+1)}\right)\left(A_{\phi(\gamma)} - I\right)e_1\varepsilon + \gamma^T\left(v_{j}^Tv_k^{(r+1)} - v_{k}^Tv_j^{(r+1)}\right)\left(A_{\phi(\gamma)} - I\right)\gamma \varepsilon^2.\end{multline*}
If $f_k^{(r+1)}(d)\cdot f_j^{(r+1)} \leq 0$, we use the same approach and obtain the following equation:
\begin{multline*}
0 = e_1^T\left(v_{j}^Tv_k^{(r+1)} + v_k^Tv_j^{(r+1)}\right)\left(A_{\phi(\gamma)} - I\right)e_1 + e_1^T\left(v_{j}^Tv_k^{(r+1)} + v_{k}^Tv_j^{(r+1)}\right)\left(A_{\phi(\gamma)} - I\right)\gamma \varepsilon\\
+ \gamma^T\left(v_{j}^Tv_k^{(r+1)} + v_{k}^Tv_j^{(r+1)}\right)\left(A_{\phi(\gamma)} - I\right)e_1\varepsilon + \gamma^T\left(v_{j}^Tv_k^{(r+1)} + v_{k}^Tv_j^{(r+1)}\right)\left(A_{\phi(\gamma)} - I\right)\gamma \varepsilon^2.\end{multline*}
Then, using the definition of $Q^{(r+1)}$, in both cases we obtain
\begin{align*} 0 &= e_1^TQ^{(r+1)}\left(A_{\phi(\gamma)} - I\right)e_1 +\left(e_1^TQ^{(r+1)}\left(A_{\phi(\gamma)} - I\right)\gamma  + \gamma^TQ^{(r+1)}\left(A_{\phi(\gamma)} - I\right) e_1\right)\varepsilon\\&\hspace{1.2cm} + \gamma^T Q^{(r+1)}\left(A_{\phi(\gamma)} - I\right)\gamma \varepsilon^2\\
&= e_1^TQ^{(r+1)}\left(A_{\phi(\gamma)} - I\right)e_1 +e_1^T\big[Q^{(r+1)}\left(A_{\phi(\gamma)} - I\right) + \left(A_{\phi(\gamma)} - I\right)^T\left(Q^{(r+1)}\right)^T\big]\gamma\varepsilon\\ &\hspace{1.2cm}+ \gamma^T Q^{(r+1)}\left(A_{\phi(\gamma)} - I\right)\gamma \varepsilon^2.\end{align*}
Edges $j$ and $k$ have the same relative exceedance for all $\varepsilon$ small enough if and only if the coefficients at $\varepsilon^0$, $\varepsilon^1$ and $\varepsilon^2$ are equal to 0. Employing this argument yields the result.
\end{proof}
Note that Lemma \ref{conditions} can be generalized in a relatively straightforward manner to hold in the cases of network disconnections. Under this scenario, it is necessary to account for the possible reduction in demand and generation at the given step of the cascade. This would only require adaptation of matrix $A_{\phi(\gamma)} - I$ in order to take into account the scaling of the demand and generation, but the structure of the proof remains the same. 

The necessary and sufficient conditions in Lemma \ref{conditions} can be difficult to verify. Therefore, in Corollary \ref{skew-symmetricCond} we show that a sufficient condition for non-uniqueness is the skew-symmetry of $Q^{(r+1)}(A_I - I)$. This is a convenient, easily-verifiable condition and can be used to find graph instances that exhibit non-uniqueness.\\
\begin{corollary} \label{skew-symmetricCond}
    Let $A\in \mathcal{A}$ and let $\Gamma(A) = \{\gamma ~:~ g^*(d) = Ad \text{ with } d = e_1 + \varepsilon \gamma \}$. If for edges $j,k\in \mathcal{E}^{(r+1)}$,  $Q^{(r+1)}(A - I)$ is skew-symmetric, i.e.
    $$\left(Q^{(r+1)}(A - I)\right)^T = -Q^{(r+1)}(A - I),$$ then the relative flow exceedances at edges $j$ and $k$ after $r$ edge failures are equal for all $\varepsilon\geq 0$ sufficiently small and all $\gamma \in \Gamma(A)$, i.e. 
    \[\psi_{j}^{(r+1)}(d) = \psi_{k}^{(r+1)}(d) ~~\text{with}~~ d = e_1 + \varepsilon \gamma.\]
\end{corollary}
The proof of this corollary follows from Lemma \ref{conditions} and properties of skew-symmetric matrices.
\begin{proof}
    Let $j,k\in \mathcal{E}^{(r+1)}$, $\gamma\in \Gamma(A)$, and suppose that $Q^{(r+1)}(A - I)$ is a skew-symmetric matrix. A well-known result for skew-symmetric matrices shows that $w^TQ^{(r+1)}(A - I)w = 0$ for all $w\in \mathbb{R}^n$. If we choose $w = e_1,$ then skew-symmetry of $Q^{(r+1)}(A - I)$ implies that $e_1^T Q^{(r+1)}(A - I)e_1 = 0$, and therefore Equation \eqref{const} holds. Similarly, for $w = \gamma$ skew-symmetry implies that $\gamma^TQ^{(r+1)}(A - I)\gamma = 0$, hence Equation \eqref{quad} holds. Finally, since $Q^{(r+1)}(A-I)$ is skew-symmetric, we have that $(A-I)^T\left(Q^{(r+1)}\right)^T = \left(Q^{(r+1)}(A - I)\right)^T = -Q^{(r+1)}(A - I)$ and therefore Equation \eqref{lin} holds as well. Hence, by Lemma \ref{conditions}, we conclude that $\psi_{j}^{(r+1)}(d) = \psi_{k}^{(r+1)}(d)$ with $d = e_1 + \varepsilon\gamma$.

    Since $\gamma$ was chosen arbitrarily, we conclude that the relative flow exceedance on edges $j$ and $k$ are equal for all $\gamma\in \gamma(A)$ and $\varepsilon\geq 0$ small enough.
\end{proof}

It turns out that, indeed, there exist graphs that yield a non-unique cascade sequence with positive probability. In what follows, we present and discuss an example of such a graph instance. However, before we can do so, we need to introduce an additional lemma. The lemma shows that for certain choices of $d$, matrix $A_{I^\ast(d)}$ has a closed-form solution, which we exploit in Example \ref{example}.
\begin{lemma}\label{solutionCd}
    If $Cd\geq 0$ for some demand vector $d\in \mathbb{R}_{\geq 0}^n$, then the optimal solution $g^\ast(d)$ to problem \eqref{DC-OPF} is given by
    \[g^\ast(d) := (1-\lambda)d + \lambda\bar{d}e.\]
     Note that the optimal solution can be also written as the following matrix equation:
    \[g^*(d) = \left((1-\lambda)I + \frac{\lambda}{n}J\right)d.\]
\end{lemma}
The proof can be found in Appendix \ref{ProofSolutionCd}. It uses similar ideas as used in the proof of Proposition \ref{propositionKKT}. It also bears similarities with the proof of \cite[Lemma IV.2.]{Nesti_2020}. With this, we now discuss the graph example.

\begin{example}
\label{example}
Consider the graph on $6$ nodes, depicted in Figure \ref{fig:directed-graph}, where node~1 has the largest demand.

We consider demands $d = e_1 + \varepsilon \gamma$, with \begin{equation}\gamma\in \tilde{S} := \left\{\gamma\in \mathbb{R}^6_{\geq0} : e^T\gamma = 1, ~\gamma_1 = 0,~ \gamma_3>\gamma_4\geq  \gamma_5\geq \gamma_2\geq \gamma_6, ~3\gamma_3>4\gamma_4 - \gamma_5 \right\},\label{setS}\end{equation} and we choose the graph orientation such that $Cd\geq 0$ for $\varepsilon$ sufficiently small. For these demands, $g^*(d) = (1-\lambda)d + \lambda\bar{d}e,$ which follows form Lemma \ref{solutionCd}. It can be shown that failure of edge (7) initiates cascade sequence $(7)\rightarrow (11) \rightarrow (10)$, after which edges (5), (6), and (9) are the maximizers of the exceedance at this point of the cascade. 

Let $V$ denote the original PTDF matrix of the graph and let $V^{(4)}$ denote the PTDF matrix of the graph after the failure of edges (7), (10), and (11). We find that
\[V = \frac{1}{30}\begin{pmatrix}
    7 & -5 & 1 & 1 & 1 & -5\\
    6 & 0 & -6 & 0 & 0 & 0\\
    6 & 0 & 0 & -6 & 0 & 0\\
    6 & 0 & 0 & 0 & -6 & 0\\
    1 & -5 & 7 & 1 & 1 & -5\\
    1 & -5 & 1 & 7 & 1 & -5\\
    1 & -5 & 1 & 1 & 7 & -5\\
    5 & 5 & 5 & 5 & 5 & -25\\
    0 & 0 & 6 & -6 & 0 & 0 \\
    0 & 0 & 6 & 0 & -6 & 0\\
    0 & 0 & 0 & 6 & -6 & 0
\end{pmatrix} ~~\text{and}~~ V^{(4)} = \frac{1}{24}\begin{pmatrix}
    6 & -6 & 0 & 0 & 6 & -6\\
    5 & -1 & -7 & -1 & 5 & -1\\
    5 & -1 & -1 & -7 & 5 & -1\\
    4 & 4 & 4 & 4 & -20 & 4\\
    1 & -5 & 7 & 1 & 1 & -5\\
    1 & -5 & 1 & 7 & 1 & -5\\
    0 & 0 & 0 & 0 & 0 & 0\\
    4 & 4 & 4 & 4 & 4 & -20\\
    0 & 0 & 6 & -6 & 0 & 0\\
    0 & 0 & 0 & 0 & 0 & 0\\
    0 & 0 & 0 & 0 & 0 & 0\\
\end{pmatrix}.\]

Next, we use Corollary \ref{skew-symmetricCond} to show that after the failure of edges (7), (10), and (11), the exceedance on edges (5), (6), and (9) is equal for all $\gamma\in \tilde{S}$. 

Let $v_i = e_i^TV$ and $v^{(4)}_i = e_i^TV^{(4)}$ for $i = 1,\dots,11$. We observe that $v^{(4)}_i = \frac{30}{24}v_i = \frac{5}{4}v_i$ for $i = 5,6,9$.  Moreover, we notice that the flow on edges $(5)$, $(6)$, and $(9)$ is positive for all $\varepsilon$ small enough. Hence, let $Q^{(4)} = v_i^Tv^{(4)}_j - v_j^Tv^{(4)}_i$ for $i\neq j$, $i,j\in \{5,6,9\}$. Then,
\[Q^{(4)} =  v_i^Tv^{(4)}_j - v_j^Tv^{(4)}_i = \frac{5}{4}v_i^Tv_j - \frac{5}{4}v_j^Tv_i,\]
\[\left(Q^{(4)}\right)^T = \frac{5}{4}v_j^Tv_i - \frac{5}{4} v_i^Tv_j = -Q^{(4)}.\]
This implies that $Q^{(4)}$ is skew-symmetric. To establish non-uniqueness, it suffices to show that $Q^{(4)}(A_I - I)$ is skew-symmetric, where $A_I = (1-\lambda)I + \frac{\lambda}{n}J$. We find that
\[Q^{(4)}(A_I - I) = -\lambda Q^{(4)} + \frac{\lambda}{n}Q^{(4)}J = -\lambda Q^{(4)},\]
since $e\in \text{Ker}(Q^{(4)})$. Finally, as $Q^{(4)}$ is skew-symmetric, it follows that $Q^{(4)}(A_I - I) = -\lambda Q^{(4)}$ is also skew-symmetric. This shows the non-uniqueness of flow exceedances in our example. \\
To summarize, we have shown that if edges $(7)$, (10), and (11) fail due to a cascade sequence, then the exceedance on edges (5), (6), and (9) are equal for all $\gamma\in S$. Hence, it remains to show that cascade sequence $(7)\rightarrow(11)\rightarrow (10)$, indeed takes place for a certain subset of $\gamma$'s and that edges (5), (6), and (9) are the maximizers of the exceedance at this point of the cascade. An interested reader will find the step-by-step computations in Appendix \ref{exampleAppendix}.
\end{example}
The example shows that the tie-breaking rule is necessary to fully define the cascade progression in certain graphs. Therefore, in the following section, we discuss the impact of the tie-breaking rule on the total cascade size. 
\section{Impact of the tie-breaking rule}\label{impact}
In this section, we discuss how the choice of the tie-breaking rule influences our results. We consider the theoretical impact from the perspective of Theorem \ref{mainThm} as well as simulation results. Finally, we show that in graphs where the tie-breaking rule does not play a role because the exceedance maximizer is unique at every step of the cascade, Theorem~\ref{mainThm} and \cite[Theorem IV.12.]{Nesti_2020} are equivalent.

Theorem \ref{mainThm} holds true for any tie-breaking rule $T$ given by Definition \ref{tie-breakingDef}. Therefore, it is natural to investigate the influence of the choice of a tie-breaking rule on the total cascade size. In the occurrence of non-uniqueness, the tie-breaking rule plays a pivotal role as it determines the cascade sequence, which, in turn, can impact the number of nodes $z\in \mathcal{Z}$ disconnected from the node with the largest demand. Consequently, the choice of the tie-breaking rule only affects the constant $C$. Hence, in general, the scale-free behavior of the total failure size prevails, independent of the choice of the rule. The only exception arises when the choice of the rule influences the network's connectivity after the cascade. In such a case, depending on the tie-breaking rule, one would observe that the tail either exhibits a scale-free behavior with parameter $\alpha$ or it scales as $\mathcal{O}(x^{-2\alpha})$.

Despite expectations coming from theoretical results, suggesting that a tie-breaking rule plays a role in the total cascade size, practical evidence seems to contradict this notion. Through a simulation experiment, we studied all possible connected graph instances up to and including 7 nodes. We found that even though the choice of the tie-breaking rule leads to different cascades, the resulting total failure size remained consistent across all choices of the tie-breaking rule. In other words, the post-cascade connected components were not affected by the tie-breaking rule. We further extended our simulation experiment to randomly chosen graph instances with more than 7 nodes and, again, we found no evidence suggesting that a tie-breaking rule influences the total cascade size in the asymptotic regime. Based on this, we state the following conjecture:
\begin{conjecture}
    The result of Theorem \ref{mainThm} holds independent of the choice of the exceedance tie-breaking rule. 
\end{conjecture}
Based on our simulation results, we believe in the validity of this conjecture. However, proving it remains a challenge for future research. Nonetheless, our experiments provide some insight into the plausible reasons behind this observed independence. We have seen that non-uniqueness typically occurs in locally dense graphs, whereby the failure of a non-unique maximizer does not lead to network disconnection. While in theory, due to flow redistribution any edge could become the next exceedance maximizer, in practice, we observe that it is the remaining maximizers that become additionally loaded and their relative exceedance continues to be the largest compared to all other edges. Thus, these remaining maximizers break in subsequent steps, creating the same connected components at the end of the cascade for each tie-breaking rule. An example of this behavior can be observed for the graph instance presented in Example \ref{example}.

Lastly, we consider a scenario in which the choice of a tie-breaking rule does not affect the cascade sequence in the neighborhood of $\varepsilon = 0$. In other words, for all $\varepsilon\geq 0$ sufficiently small, all exceedance maximizers are unique. Note that in this scenario, \cite[Theorem IV.12.]{Nesti_2020} holds. Since our main result generalizes this theorem, both theorems should yield the same result. In the following lemma, we show that this is indeed true.
\begin{lemma} \label{equivThm}
Let $\mathcal{G} = (\mathcal{N},\mathcal{E})$ and $\lambda\in (0,1)$ satisfy the following conditions:
\begin{enumerate}
    \item For all edges $j$, the ratios between redistributed flows and edge limits 
    \[\frac{f_j^{(r)}(e_1)}{F_j(e_1)} = \frac{\lambda |(V^{(r)}e_1)_j|}{|(Ve_1)_j|}\]
    are all different for all $r\geq 2$ (Assumption IV.4 of \cite{Nesti_2020}).
    \item For all edges $j$ and $r\geq 2$,
    \[|f_j^{(r)}(e_1)| - F_j(e_1) \neq 0. \]
    (Assumption IV.7 of \cite{Nesti_2020}).
\end{enumerate}
Then, if there exist $i\in \mathcal{N}$ and $l\in \mathcal{E}$ for which the cascade results in network disconnection, 
    \[\mathbb{P}(S>x)\sim Cx^{-\alpha} \text{ as } x\rightarrow \infty,\]
with  $$C := \sum_{i\in \mathcal{N}}\sum_{l\in\mathcal{E}}\frac{K}{m}\left(Z(i,l)\cdot \frac{\lambda}{n}\right)^{\alpha},$$
where $Z(i,l)$ is the number of cities disconnected from node $i$ after a cascade initiated by edge $l$ and demand $d = e_i$. Otherwise, $\mathbb{P}(S>x) = \mathcal{O}(x^{-2\alpha})$.
\end{lemma}
\begin{proof}
    The two assumptions of the theorem ensure that for all $d = e_1 + \varepsilon \gamma$ and $\varepsilon>0$ sufficiently small, the cascade sequence is the same and equal to the cascade sequence under demand $d = e_1$ \cite[Proposition IV.9]{Nesti_2020}. This implies that for all $\varepsilon$ small enough, the result of the cascade is the same for all $\gamma\in \Gamma$. In particular, for all $\gamma\in\Gamma$, $Z(i,l,\gamma) = Z(i,l)$. Hence,
    \begin{align*}
        &\sum_{I^\ast\in \mathcal{I}}\sum_{z\in \mathcal{Z}} z^\alpha \cdot \mathbb{P}(Y\in Y(i,l,z,I^\ast) | N = i, L^{(1)} = l) \\&\hspace{0.6cm}= Z(i,l)^\alpha \sum_{I^\ast\in \mathcal{I}}\mathbb{P}(Y\in Y(i,l,Z(i,l),I^\ast) | N = i, L^{(1)} = l) \\&\hspace{0.6cm}= Z(i,l)^\alpha.
    \end{align*}
    Thus, the result of Theorem \ref{mainThm} simplifies to
    \[\mathbb{P}(S>x) \sim \sum_{i\in \mathcal{N}}\sum_{l\in \mathcal{E}}\frac{K}{m}\left(Z(i,l)\cdot \frac{\lambda}{n}\right)^\alpha  x^{-\alpha}\]
    in case of network disconnection. 
    
    If there exist no $i$ and $l$ that lead to network disconnection when $\sum_{j = 2}^n Y_j\leq \varepsilon Y_1$, then $$\mathbb{P}(S>x) = \mathbb{P}(S>x, \sum_{j = 2}^n Y_j> \varepsilon Y_1) = \mathcal{O}(x^{-2\alpha}).$$
\end{proof}
To summarize, the tie-breaking rule allows us to extend the existing results on the tail of the total cascade size to all simple, connected graphs. In theory, the choice of the tie-breaking rule can impact the constant $C$, describing the tail-scaling of the total size; however, in all of our simulation experiments, the constant $C$ was unaffected by the choice of the rule. 
\section{Conclusions} \label{conclusions}


Cascading failures in complex networks can lead to significant disruptions and disturbances across diverse domains. This paper addresses the occurrence of heavy-tailed cascading failures in overload networks, building upon recent developments in the field that suggest that heavy-tailed failures may be induced by the scale-free demand distribution. We emphasize that this is one of many paradigms that may lead to scale-free failures in such networks. Depending on the underlying problem other factors, such as criticality, may still be prominent causes of the observed scale-free behavior. To understand the impact of criticality, it would be necessary to study this problem from the perspective of growing networks. This approach has not been pursued so far and remains a subject for future research.


Our contributions include a novel characterization of optimal nodal generation that arises as the optimal solution to the DC-OPF problem, necessary and sufficient conditions for non-uniqueness of exceedance maximizers, and a theorem on the heavy-tailed nature of total failure size for all simple and connected graph topologies.

In our model, failure propagation is governed by edges maximizing the flow exceedance. When the maximizer is not unique, we choose the next edge to break based on some tie-breaking rule. While the tie-breaking rule theoretically influences failure propagation, our simulation experiment suggests its independence from the total cascade size. Formal proof of this independence requires further investigation. 

It is crucial to recognize that the fundamental principles behind the scale-free failure size are versatile and can be applied to various cascading failure scenarios. We anticipate that as long as the heavy-tailed mismatch between generation and demand in a cascading failure model is maintained, the failure size will be heavy-tailed. The development of a universal cascading failure model for diverse scenarios beyond the power networks application would be an interesting future research direction. In addition, investigating the impact of network interdependencies on cascade size is crucial. Many real-world systems exhibit intricate interconnectedness, amplifying failures across networks. Integrating multilayer network theory with our model promises deeper insights into cascading failures within these complex systems.

\section*{Acknowledgments}
This work is supported by NWO through Gravitation NETWORKS grant no. 024.002.003.

\bibliographystyle{unsrt}  
\bibliography{references}  

\appendix
\section{Appendix: Remaining proofs} \label{appendixProofs}
In this section, we provide all remaining proofs of results in this paper.

\subsection{Proof of Lemma \ref{lemmaProjection}}\label{proofLemmaProjection}
The proof of this lemma requires well-known results from the theory of orthogonal projection, which we briefly introduce beforehand. Consider a non-empty affine subspace $S :=\{x\in \mathbb{R}^n:Mx = b\}$, for some $M \in \mathbb{R}^{r\times n}$ with rank $r\leq n$ and $b\in \mathbb{R}^r$. 
An orthogonal projection of point $z$ onto $S$ is given by \cite{orthProj,plesnik2007}
\begin{align*}
    p(z):= \Big(I - M^T\big(MM^T\big)^{-1}M\Big)z + M^T\big(MM^T\big)^{-1}b.
\end{align*}
The projection equation can be generalized to the situations when $M$ is not a full-rank matrix. In that case, the inverse operator is replaced with the pseudoinverse operator $+$. In particular,
\begin{align}
    p(z):= \Big(I - M^T\big(MM^T\big)^{+}M\Big)z + M^T\big(MM^T\big)^{+}b.
\label{projectionGeneral}\end{align}
With this, we are ready to prove the lemma.

\begin{proof}

We start by showing that $g^*(d)$ is a unique point on the boundary of $\mathcal{F}_d$. The feasible region $\mathcal{F}_d$ is defined as a finite intersection of hyperplanes, which implies that it is a convex polytope \cite{grunbaum1967convex}. From the definition of $\mathcal{F}_d$, see \eqref{Fd}, it follows that the feasible region is closed. Moreover, it is non-empty because $d\in \mathcal{F}_d$. Finally, $\bar{d}e\not\in \mathcal{F}_d$, as it is the solution to the unconstrained DC-OPF problem, and therefore, by construction, $\bar{d}e$ does not satisfy the operational flow constraints, given in \eqref{edgeConstraints}. Specifically, since $e\in \text{Ker}(V)$, the constraint $-\lambda Vd \leq V(d - \bar{d}e) = Vd \leq \lambda Vd$ is violated. Therefore, in light of Lemma \ref{AltOptSol}, we can conclude that the optimal solution $g^*(d)$ is unique and lies on the boundary of $\mathcal{F}_d$ \cite[Ch. 8]{boyd2004convex}.

By the definition given in \eqref{Istar}, $S_{I^*(d)}$ is the affine space on which $g^*(d)$ lies. Hence, we can use projection theory to derive $g^*(d)$ as the projection of $\bar{d}e$ onto $S_{I^*(d)}$. Recall that $S_{I^*(d)}= \bigcap_{i\in I^*(d)} H_i$, which is equivalently given by $S_{I^*(d)} = \{g\in \mathbb{R}^n ~:~ M_Ig = b_I\},$
with $M_I:= \begin{bmatrix}
    e^T\\V_{I_1}\\V_{I_2} 
\end{bmatrix}\in \mathbb{R}^{(s_1 +s_2 + 1)\times n}$ and $b_I:= \begin{bmatrix}
    e^T\\(1-\lambda)V_{I_1}\\(1+\lambda)V_{I_2}
\end{bmatrix}d\in \mathbb{R}^{s_1 +s_2 + 1}.$

According to $\eqref{projectionGeneral}$, the projection of $\bar{d}e$ onto $S_{I^*(d)}$ is given by
\[p_{I^*(d)}(\bar{d}e) := \Big(I - M_{I}^T\big(M_{I}M_{I}^T\big)^{+}M_{I}\Big)\bar{d}e + M_{I}^T\big(M_{I}M_{I}^T\big)^{+}b_{I}.\]
Our goal is to show that $\Big(I - M_{I}^T\big(M_{I}M_{I}^T\big)^{+}M_{I}\Big)\bar{d}e = 0$, which would imply that $p_{I^*(d)} = M_{I}^T\big(M_{I}M_{I}^T\big)^{+}b_{I}$. Then, we will rewrite this expression to obtain $A_{I^*(d)}$, as given in the lemma.\\
First, we find that
\[M_{I}\bar{d}e = \bar{d}\begin{bmatrix} e^T \\
V_{I_1}\\
V_{I_2}\end{bmatrix}e = \bar{d} e^Tee_1 = n\bar{d}e_1\]
because $\text{Ker}(V) = \langle e\rangle$. Using this property again, we find that
\[M_{I}M_{I}^T = \begin{bmatrix} e^T \\
V_{I_1}\\
V_{I_2}\end{bmatrix}\begin{bmatrix} e &
V^T_{I_1}&
V^T_{I_2}\end{bmatrix} =\left[\begin{array}{c|c} e^Te & e^T\begin{bmatrix}
V^T_{I_1}&
V^T_{I_2}\end{bmatrix} \\[4pt]
\hline\\[-6pt]
 \vspace{0.1cm} 
\begin{bmatrix}
V_{I_1}\\
V_{I_2}\end{bmatrix}e & \begin{bmatrix}
V_{I_1}\\
V_{I_2}\end{bmatrix}\begin{bmatrix}
V^T_{I_1}&
V^T_{I_2}\end{bmatrix}\end{array}\right] =\left[\begin{array}{c|c}
  n    & 0 \\[4pt]
  \hline\\[-6pt]
 \vspace{0.1cm} 
   0  &  \begin{bmatrix}
V_{I_1}\\
V_{I_2}\end{bmatrix}\begin{bmatrix}
V^T_{I_1}&
V^T_{I_2}\end{bmatrix}
\end{array}\right]. \]
Next, we compute $(M_{I}M_{I}^T)^{+}e_1$. It is a well-known property that
\[\left[\begin{array}{c|c}
  D_1   & 0 \\[3pt]
  \hline\\[-8pt]

   0  & D_2
\end{array}\right]^{+} = \left[\begin{array}{c|c}
  D_1^{+}   & 0  \\[3pt]
  \hline\\[-8pt]
   0  & D_2^{+}
\end{array}\right].\]
Therefore, we find
\[(M_IM_I^T)^{+}e_1 = \left[\begin{array}{c|c}
  \frac{1}{n}    & 0  \\[3pt]
  \hline\\[-8pt]
 \vspace{0.1cm} 
   0  &  \left(\begin{bmatrix}
V_{I_1}\\
V_{I_2}\end{bmatrix}\begin{bmatrix}
V^T_{I_1}&
V^T_{I_2}\end{bmatrix}\right)^{+}
\end{array}\right]e_1= \frac{1}{n}e_1.\]
Hence,
\[M_{I}^T(M_{I}M_{I}^T)^{+}M_{I}\bar{d}e = n\bar{d}M_{I}^T(M_{I}M_{I}^T)^{+}e_1 = \bar{d}M_{I}^Te_1 = \bar{d}e.\]
So, altogether, we have found that \[\Big(I - M_{I}^T\big(M_{I}M_{I}^T\big)^{+}M_{I}\Big)\bar{d}e = 0.\]
Hence,
\[p_{I^*(d)}(\bar{d}e) = M_{I}^T\big(M_{I}M_{I}^T\big)^{+}\begin{bmatrix}e^T\\
(1-\lambda)V_{I_1} \\
(1+\lambda)V_{I_2}\end{bmatrix}d.\]
Manipulating the above matrix expression yields 
\begin{align*} &M_{I}^T(M_{I}M_{I}^T)^{+}\begin{bmatrix}e^T\\
(1-\lambda)V_{I_1} \\
(1+\lambda)V_{I_2}\end{bmatrix}
\\&\hspace{0.6cm}= \begin{bmatrix}e & V_{I_1}^T & V_{I_2}^T\end{bmatrix}\left[\begin{array}{c|c}
  \frac{1}{n}    & 0  \\[3pt]
  \hline\\[-8pt]
 \vspace{0.1cm} 
   0  &  \left(
\begin{bmatrix}V_{I_1}\\ V_{I_2}\end{bmatrix}
\begin{bmatrix}V^T_{I_1} & V^T_{I_2}\end{bmatrix}\right)^{+}
\end{array}\right]\begin{bmatrix}e^T\\
(1-\lambda)V_{I_1} \\
(1+\lambda)V_{I_2}\end{bmatrix}\\&\hspace{0.6cm}= \begin{bmatrix}\frac{1}{n}e & \begin{bmatrix} V_{I_1}^T & V_{I_2}^T\end{bmatrix}\left(
\begin{bmatrix}V_{I_1}\\ V_{I_2}\end{bmatrix}
\begin{bmatrix}V^T_{I_1} & V^T_{I_2}\end{bmatrix}\right)^{+}\end{bmatrix}\begin{bmatrix}e^T\\
(1-\lambda)V_{I_1} \\
(1+\lambda)V_{I_2}\end{bmatrix}\\&\hspace{0.6cm}= \frac{1}{n}J + \begin{bmatrix} V_{I_1}^T & V_{I_2}^T\end{bmatrix}\left(
\begin{bmatrix}V_{I_1}\\ V_{I_2}\end{bmatrix}
\begin{bmatrix}V^T_{I_1} & V^T_{I_2}\end{bmatrix}\right)^{+}\begin{bmatrix}
(1-\lambda)V_{I_1} \\
(1+\lambda)V_{I_2}\end{bmatrix}\\
&\hspace{0.6cm}= A_{I^*(d)}.\end{align*}
\end{proof}
\subsection{Proof of Lemma \ref{fullLine}} \label{proofFullLine}
The proof of this lemma relies on a contradiction argument that arises from a construction of an alternative solution to the DC-OPF problem. 
\begin{proof}

 Let $\varepsilon>0$ be sufficiently small and let $g^\ast(d)$ be the optimal solution to the optimization problem \eqref{DC-OPF}. In this proof, we argue by contradiction that \eqref{fullLines} holds.
 
For the sake of contradiction, suppose that there exists some $i\in \mathcal{E}_1$ for which $$(1-\lambda)(Vd)_i<(Vg^\ast(d))_i.$$
Let $x>0$ denote the total unused capacity on edges in $\mathcal{E}_1$, i.e.
\[x:=  \sum_{i\in \mathcal{E}_1} \lambda \big(Vd\big)_i - \big(V(d - g^\ast(d))\big)_i.\] 

Recall that the solution to the planning problem is given by $g^{(pl)}(d) = \frac{1+\varepsilon}{n}e$. The total flow on edges in $\mathcal{E}_1$ during the \textit{planning} problem is equal to the unsatisfied demand at node 1, which is $1 - \frac{1+\varepsilon}{n}$. Since the \textit{operational} edge limits are set to fraction $\lambda$ of flows in the \textit{planning} phase, the total \textit{operational} capacity of flow on edges in $\mathcal{E}_1$ is equal to $\lambda (1 - \frac{1+\varepsilon}{n} )$. This means that the generation of node $1$ is equal to $1 - \left(\lambda (1 - \frac{1+\varepsilon}{n}) - x \right) = 1 - \lambda +  \frac{1+\varepsilon}{n}\lambda + x.$ Because the objective function is a sum of squares of generations at each node, the objective function is minimal if the remaining generation in the network is equally distributed over all other nodes. In particular, each remaining node generates fraction $\frac{1}{n-1}$ of the remaining generation, which is given by $  (1-\lambda)\varepsilon + (n-1)\frac{1+\varepsilon}{n} \lambda - x$. Therefore, the value of the objective function corresponding to $g^\ast(d)$ is at least
\[\frac{1}{2}\sum_{i = 1}^n g^\ast_i(d)^2\geq \frac{1}{2}(1 - \lambda +  \frac{1+\varepsilon}{n}\lambda + x)^2 + \frac{n-1}{2}\left(\frac{1-\lambda}{n-1}\varepsilon + \frac{1+\varepsilon}{n} \lambda - \frac{x}{n-1}\right)^2.\]
Now, consider an alternative solution $\bar{g}(d) = (1-\lambda)d + \frac{1+\varepsilon}{n}\lambda e$. This solution is feasible as $\sum_{i = 1}^n \bar{g}_i(d) = 1+\varepsilon$ and $V(d - \bar{g}(d)) = \lambda Vd$, i.e., the corresponding flow on each edge is equal to the operational capacity limit. We find that the value of the objective function corresponding to this generation is at most 
\[\frac{1}{2}\sum_{i = 1}^n \bar{g}_i(d)^2 \leq  \frac{1}{2}(1 - \lambda +  \frac{1+\varepsilon}{n}\lambda)^2 + \frac{n-1}{2}\left((1-\lambda)\varepsilon + \frac{1+\varepsilon}{n} \lambda \right)^2.\]
With this, we find that, in the limit $\varepsilon\downarrow 0$, the difference in the values of the objective function for $g^\ast(d)$ and $\bar{g}(d)$ is at least
\begin{align*}
   &\lim_{\varepsilon \downarrow 0} \frac{1}{2}\left(\sum_{i = 1}^n g^\ast_i(d)^2 - \sum_{i = 1}^n \bar{g}_i(d)^2 \right)\\&\hspace{0.6cm}\geq \frac{1}{2}\left(1 - \lambda + \frac{\lambda}{n}\right)^2 + \frac{1}{2}x^2 + x\left(1 - \lambda + \frac{\lambda}{n}\right) + \frac{(n-1)}{2n^2}\lambda^2 \\&\hspace{1.2cm}+ \frac{x^2}{2(n-1)}  - \frac{x}{n}\lambda  - \frac{1}{2}\left(1 - \lambda + \frac{\lambda}{n}\right)^2- \frac{(n-1)}{2n^2}\lambda^2 \\&\hspace{0.6cm}= \frac{1}{2}\left(\frac{n}{n-1}x^2 \right)+ (1-\lambda)x.
   \end{align*}
We observe that since $x>0$, the limit is positive. Moreover, functions $\frac{1}{2}\sum_{i = 1}^n \bar{g_i}(d)^2$ and $\frac{1}{2}\sum_{i = 1}^n g_i^*(d)^2$ are polynomials in $\varepsilon$ and therefore continuous. Hence, for all $\varepsilon$ sufficiently small, $\frac{1}{2}\sum_{i = 1}^n \bar{g_i}(d)^2< \frac{1}{2}\sum_{i = 1}^n g_i^*(d)^2$ . This is a contradiction because we have assumed that $g^\ast(d)$ is optimal. 
\end{proof}
\subsection{Proof of Lemma \ref{tightConstraints}} \label{ProofTightConstraints}
The following lemma is proven using the continuity of the optimal solution of the DC-OPF problem with respect to the demand vector $d$.
\begin{proof}
    First, we observe that the inequality is true for the limiting case of $\varepsilon = 0$. That is, in this case, the optimal solution is given by $g^\ast(e_1) = (1-\lambda)e_1 + \frac{\lambda}{n}e$ \cite{Nesti_2020}. If $i$ is such that $(Ve_1)_i>0$, then,
    \[(Vg^\ast(e_1))_i = (1-\lambda)(Ve_1)_i + \frac{\lambda}{n}(Ve)_i = (1-\lambda)(Ve_1)_i < (1+\lambda) (Ve_1)_i.  \]
    In addition, we know that the solution $g^\ast(d)$ of the DC-OPF problem is continuous in $d$, as this problem is an example of multi-parametric quadratic programming with a strictly convex objective. For this class of problems, the continuity is a known result \cite[Theorem 1]{TONDEL2003489}. The continuity implies that there exists some function $f(\varepsilon, \gamma)$ with $\lim_{\varepsilon\rightarrow 0} f(\varepsilon, \gamma) = 0$ such that
    \[g^\ast(e_1 + \varepsilon \gamma) = g^\ast(e_1) + f(\varepsilon, \gamma).\]
    Hence, for all $\varepsilon>0$ sufficiently small, we have 
     \[(Vg^\ast(d))_i = \left(Vg^\ast(e_1)\right)_i + \left(Vf(\varepsilon,\gamma)\right)_i = (1-\lambda)(Ve_1)_i + \left(Vf(\varepsilon,\gamma)\right)_i < (1+\lambda) (Ve_1)_i.\]
    \end{proof}

\subsection{Proof of Lemma \ref{solutionCd}} \label{ProofSolutionCd}
The proof of this lemma relies on the KKT conditions and it uses similar ideas as the proof of Proposition \ref{propositionKKT}.
\begin{proof}
    Let $d\in \mathbb{R}^n_{\geq 0}$ be a demand vector such that $Cd\geq 0$. We prove that $g^\ast(d) = (1-\lambda)d + \lambda\bar{d}e$ is the optimal solution by showing that it satisfies the KKT conditions \eqref{eq1} -- \eqref{eq6}. 

    Let $\nu = 0$, $\mu = (1-\lambda)Cd$ and $\delta = -\bar{d}$. This choice of $\nu$ automatically satisfies \eqref{eq4}. Moreover, since $Cd\geq 0$ and $\lambda\in (0,1)$, we have that $\mu\geq 0$. This means that \eqref{eq2} is also satisfied for our choice of $\mu$. \\
    
    Next, we show that $e^Tg^\ast(d) = e^Td$ (Condition \eqref{eq6}). We find
    \[e^Tg^\ast(d) = e^T\left((1-\lambda)d + \lambda \bar{d}e\right) = (1-\lambda)n\bar{d} + \lambda n \bar{d} = n\bar{d} = e^Td.\]
    Furthermore, we find that both Conditions \eqref{eq5} and \eqref{eq3} hold because
    \[Vg^\ast(d) = (1-\lambda)Vd + \lambda \bar{d}Ve = (1-\lambda)Vd.\]
    Hence, it remains to show that Condition \eqref{eq1} is satisfied; specifically, that
    \begin{equation}g^\ast(d) + V^T(\nu - \mu) + \delta e = 0.\label{KKT1}\end{equation}
    Using the fact that $V= CL^+$, $L = C^TC$, $\left(L^+\right)^T = L^+$ and $L^+L = (I - \frac{1}{n}J)$ \cite{Nesti_2020}, we find
    \[V^T\mu = (1-\lambda)V^TCd = (1-\lambda)\left(L^+\right)^TC^TCd = (1-\lambda)L^+Ld = (1-\lambda)(I - \frac{1}{n}J)d = (1-\lambda)d - (1-\lambda)\bar{d}e.\]
    We substitute the expressions for $g^\ast(d)$, $\mu$, $\nu$ and $\delta$ into \eqref{KKT1}. This yields
    \[(1-\lambda)d + \lambda\bar{d}e - (1-\lambda)d + (1-\lambda)\bar{d}e - \bar{d}e = 0.\]
    This means that Condition \eqref{eq1} holds as well. With this, we conclude that all KKT conditions are satisfied, which proves that $g^\ast(d)$ is the optimal solution. 
\end{proof}

\section{Appendix: Cascade propagation in Example \ref{example}} \label{exampleAppendix}

In this section, we present the step-by-step computations regarding the cascade propagation on the graph example discussed in Section \ref{emergency}. The graph is also visualized in Figure \ref{graphAppendix}.

\begin{figure}[h]
\centering
\begin{tikzpicture}[ scale=0.6]
\begin{scope}[every node/.style={circle,thick,draw}]
    \node[circle,thick,draw, fill = lightgray] (1) at (0,-1.1) {1};
    \node (2) at (5,-4.6) {2};
    \node (3) at (3,-10) {3};
    \node (4) at (-3,-10) {4};
    \node (5) at (-5,-4.6) {5};
    \node (6) at (10,-4.6) {6};
\end{scope}

\begin{scope}[>={Stealth[black]},
              every node/.style={fill=white,circle},
              every edge/.style={draw=black,very thick}]
    \path [<-] (1) edge node {$(1)$} (2);
    \path [<-] (1) edge node {$(2)$} (3);
    \path [<-] (1) edge node {$(3)$} (4);
    \path [<-] (1) edge node {$(4)$} (5);
    \path [->] (2) edge node {$(5)$} (3);
    \path [->] (2) edge node {$(6)$} (4);
    \path [<-] (2) edge node {$(8)$} (6);
    \path [->] (4) edge node {$(9)$} (3);
    \path [<-] (3) edge node {$(10)$} (5);
    \path [<-] (4) edge node {$(11)$} (5);
    \path [<-] (5) edge node {$(7)$} (2);
\end{scope}

\end{tikzpicture}

\caption{Example of a graph where non-unique relative flow exceedance occurs with positive probability.}
\label{graphAppendix}

\end{figure}
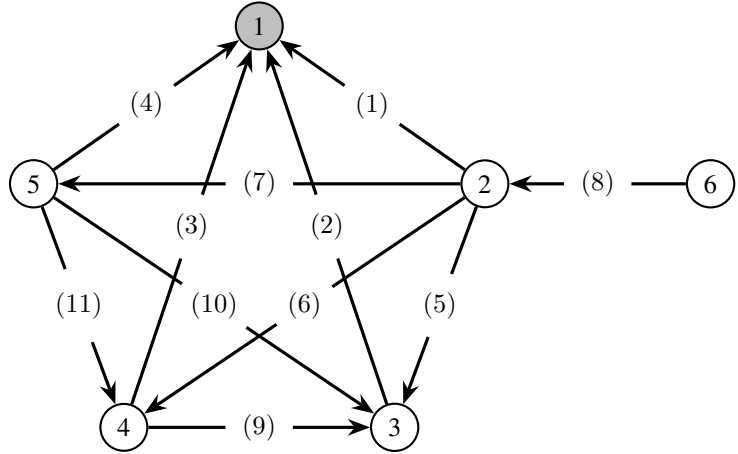


Recall that in Section \ref{example}, we have shown that after edges (7), (11), and (10) are removed from the graph due to the cascades, the relative flow exceedances on edges (5) (6) and (9) are equal for all $\gamma \in \tilde{S}$ as defined in \eqref{setS}. Here, we show that the cascade $(7)\rightarrow (11)\rightarrow (10)$ happens with positive probability and that edges (5), (6), and (9) are all maximizers of the relative exceedance.

Suppose that the initial failure occurs at edge (7). Then, the corresponding PTDF matrix $V^{(2)}$ is given by
\[V^{(2)} = \frac{1}{90}\begin{pmatrix}
    22 & -20 & 4 & 4 & 10 & -20\\
    18 & 0 & -18 & 0 & 0 & 0\\
    18 & 0 & 0 & -18 & 0 & 0\\
    17 & 5 & -1 & -1 & -25 & 5\\
    4 & -20 & 22 & 4 & 10 & -20\\
    4 & -20 & 4 & 22 & 10 & -20\\
    0 & 0 & 0 & 0 & 0 & 0 \\
    15 & 15 & 15 & 15 & 15 & -75\\
    0 & 0 & 18 & -18 & 0 & 0\\
    -1 & 5 & 17 & -1 & -25 & 5\\
    -1 & 5 & -1 & 17 & -25 & 5
\end{pmatrix}.\]
We observe that rows 2,3,8, and 9 of matrix $V^{(2)}$ remain unchanged compared to the matrix $V$. Therefore, we can immediately conclude that no exceedance occurs on edges $(2)$, (3), (8), and (9). To find the exceedance maximizer, we only need to consider the remaining edges. We obtain the following approximation:
\[\psi_1^{(2)} = \left|\frac{v^{(2)}_1(A_I - I)d}{\lambda^\ast v_1d}\right| = \frac{1}{3}\left|\frac{-\lambda(22 + \varepsilon(-20\gamma_2 + 4\gamma_3 + 4 \gamma_4 + 10 \gamma_5 - 20 \gamma_6)) }{\lambda^\ast (7 + \varepsilon( -5 \gamma_2  +\gamma_3 + \gamma_4 + \gamma_5 - 5\gamma_5)) }\right|\approx \frac{22\lambda}{21\lambda^\ast},\]
for $\varepsilon$ small. Note that the approximation is exact for $\varepsilon = 0$. Taking the same approach, we find that
\[\psi_4^{(2)}\approx \frac{17\lambda}{18\lambda^\ast}, ~~\psi_5^{(2)}\approx \frac{4\lambda}{3\lambda^\ast}, ~~ \psi_6^{(2)}\approx \frac{4\lambda}{3\lambda^\ast}.\]
Finally, we observe that 
\[\psi_{10}^{(2)} = \frac{\lambda}{3\lambda^\ast}\left|\frac{-\frac{1}{\varepsilon} + 5\gamma_2 +17\gamma_3 -\gamma_4 - 25\gamma_5 + 5\gamma_6}{6\gamma_3 - 6 \gamma_5}\right| = \frac{\lambda}{18\lambda^\ast}\frac{\frac{1}{\varepsilon} - 5\gamma_2 - 17\gamma_3 + \gamma_4 + 25\gamma_5 - 5\gamma_6}{\gamma_3 - \gamma_5},\]\[\psi_{11}^{(2)} = \frac{\lambda}{3\lambda^\ast}\left|\frac{-\frac{1}{\varepsilon}+ 5\gamma_2 -\gamma_3 + 17\gamma_4 - 25\gamma_5 + 5\gamma_6}{6\gamma_4 - 6 \gamma_5}\right| = \frac{\lambda}{18\lambda^\ast}\frac{\frac{1}{\varepsilon} - 5\gamma_2 +\gamma_3 -17 \gamma_4 + 25\gamma_5 - 5\gamma_6}{\gamma_4 - \gamma_5}.\]
Note that we change the sign of the nominator in these expressions, as for all $\varepsilon$ small enough, the term $\frac{1}{\varepsilon}$ dominates the remaining terms and $\gamma_3\geq \gamma_4\geq \gamma_5$ for all $\gamma\in \tilde{S}$. With this, we obtain the following approximations:
\[\psi^{(2)}_{10}\approx \frac{\lambda}{18\lambda^\ast\varepsilon(\gamma_3 - \gamma_5)} ~~\text{ and } ~~\psi^{(2)}_{11}\approx \frac{\lambda}{18\lambda^\ast\varepsilon(\gamma_4 - \gamma_5)}.\]
Based on the above approximations, we can conclude that for all $\varepsilon$ small enough, either edge $(10)$ or edge (11) has the largest exceedance. Moreover, since $\gamma_3>\gamma_4$ it holds that
\[\gamma_3 - \gamma_5>\gamma_4 - \gamma_5\]
and \[-17\gamma_3 + \gamma_4<\gamma_3-17 \gamma_4.\]
Together, these two conditions lead to $\psi_{10}^{(1)}<\psi_{11}^{(1)}$, indicating that edge (11) is the next one to fail in this cascade sequence.\\
Next, let $V^{(3)}$ be the PTDF matrix of the graph after the failure of edges (7) and (11). Then,
\[V^{(3)} = \frac{1}{240}\begin{pmatrix}
    59 & -55 & 11 & 5 & 35 & -55\\
    48 & 0 & - 48 & 0 & 0 & 0 \\
    49 & -5 & 1 & -65 & 25 & -5\\
    44 & 20 & -4 & 20 & -100 & 20\\
    11 & -55 & 59 & 5 & 35 & -55\\
    10 & -50 & 10 & 70 & 10 & -50\\
    0 & 0 & 0 & 0 & 0 & 0\\
    40 & 40 & 40 & 40 & 40 & -200\\
    1 & -5 & 49 & -65 & 25 & -5 \\
    -4 & 20 & 44 & 20 & -100 & 20\\
    0 & 0 & 0 & 0 & 0 & 0
\end{pmatrix}.\]
Again, we notice that the flow on edges (2) and (8) remains unchanged, resulting in no exceedance occurs on these two edges. For the remaining edges, we take the same approach as before and find the following approximations for the exceedances:
\[\psi^{(3)}_1 \approx \frac{59\lambda}{56\lambda^\ast}, ~~\psi^{(3)}_3 \approx \frac{49\lambda}{48\lambda^\ast},~~ \psi^{(3)}_4 \approx \frac{44\lambda}{48\lambda^\ast},~~ \psi^{(3)}_5 \approx \frac{11\lambda}{8\lambda^\ast}, ~~\psi^{(3)}_6 \approx \frac{10\lambda}{8\lambda^\ast},\]
\[ \psi^{(3)}_9 \approx \frac{\lambda}{48\lambda^\ast\varepsilon(\gamma_3 - \gamma_4)},~~ \psi^{(3)}_{10} \approx \frac{4\lambda}{48\lambda^\ast\varepsilon(\gamma_3 - \gamma_5)}.\]
Clearly, for $\varepsilon$ small enough, the maximal exceedance is obtained at edge (9) or (10). Since we assume that
$3\gamma_3> 4\gamma_4 - \gamma_5,$
we obtain that $4\gamma_3 - 4 \gamma_4 > \gamma_3 - \gamma_5$, from which it directly follows that $\psi_9^{(3)}<\psi_{10}^{(3)}$. Hence, the next edge to fail is (10). \\

To summarize, we have shown that for $\gamma \in \tilde{S}$, the edges of the graph will break in the following order:
\[(7)\rightarrow (11) \rightarrow (10).\]
Furthermore, we observe that the $\tilde{S}$ has a positive Lebesgue measure due to its non-empty interior. For example, consider $\gamma = (0,\frac{10}{100}, \frac{30}{100}, \frac{28}{100},\frac{27}{100}, \frac{5}{100}),$ which lies in the interior of $\tilde{S}$ because $\gamma_3 >\gamma_4$ and $3\gamma_3 = \frac{90}{100}> \frac{85}{100} = 4\gamma_4 - \gamma_5$. Since $\gamma$ is a realization of $n-1$ (scaled) i.i.d. Pareto distributions with support equal to $(0,\infty)$, the non-empty interior of $\tilde{S}$ implies that 
\[\mathbb{P}(\gamma\in \tilde{S})>0.\]
Lastly, we show that after the three edge failures have occurred, the maximizers of the exceedance are edges (5), (6), and (9). By following the same procedure as before, we compute the (approximate) exceedances for each edge and find
\[\psi_1^{(4)} \approx \frac{30\lambda}{28\lambda^\ast}, ~\psi_2^{(4)} \approx \frac{25\lambda}{24\lambda^\ast}, ~\psi_3^{(4)} \approx \frac{25\lambda}{24\lambda^\ast},~ \psi_4^{(4)} \approx \frac{20\lambda}{24\lambda^\ast},~\psi_8^{(4)} = \frac{\lambda}{\lambda^\ast},\]
\[\psi_5^{(4)} = \psi_6^{(4)} = \psi_9^{(4)} = \frac{5\lambda}{4\lambda^\ast}.\]
We observe that the above fraction is largest for edges (5), (6), and (9). Furthermore, for $\lambda^* \in [\lambda, \frac{5}{4}\lambda)$, the exceedance
on these edges is larger than 1, implying that these edges should fail next. This proves that the maximizer of the exceedance is non-unique for the provided graph and $\gamma\in \tilde{S}$, as we claimed in Example \ref{example}.
\end{document}